\documentclass[12pt]{amsart}

\usepackage[utf8]{inputenc}
\usepackage[T1]{fontenc}
\usepackage{eucal}
\usepackage{microtype}
\usepackage{csquotes}
\usepackage{mathtools}
\usepackage{amssymb,amsfonts,stmaryrd,amsthm,mathtools}
\usepackage{mathrsfs}  
\usepackage{tikz}
\usetikzlibrary{cd}
\usepackage{enumitem}
\usepackage[normalem]{ulem}
\setlist[enumerate,1]{label={\roman*)}}
\usepackage{graphicx,subcaption}

\newcommand{\mailto}[1]{\href{mailto:#1}{#1}}

\newcommand{\X}{\mathfrak{X}}

\DeclarePairedDelimiter{\ray}{\langle}{\rangle_+}

\DeclarePairedDelimiter{\sqbracks}{[}{]}
\newcommand*{\lieBr}[1]{\sqbracks{#1}}
\DeclarePairedDelimiter{\jacBr}{\lbrace}{\rbrace}
\DeclarePairedDelimiter{\set}{\{}{\}}
\DeclarePairedDelimiter{\gen}{\langle}{\rangle}
\newcommand{\Cinfty}{\mathscr{C}^\infty}

\newcommand{\JacLambda}{\Lambda} 
\newcommand{\lsharp}{\sharp_\JacLambda}

\usepackage[left=2.5cm,right=2.5cm,top=3cm,bottom=3.5cm]{geometry}

\theoremstyle{plain}
\newtheorem{theorem}{Theorem}
\newtheorem*{theorem*}{Theorem}
\newtheorem{lemma}[theorem]{Lemma}
\newtheorem*{lemma*}{Lemma}
\newtheorem{proposition}[theorem]{Proposition}
\newtheorem*{proposition*}{Proposition}
\newtheorem{corollary}[theorem]{Corollary}  
\newtheorem*{corollary*}{Corollary}
\theoremstyle{definition}
\newtheorem{definition}{Definition}
\newtheorem*{definition*}{Definition}

\newtheorem*{example*}{Example}
\newtheorem{remark}{Remark}
\newtheorem*{remark*}{Remark}

\newtheorem*{conjecture*}{Conjecture}

\newtheorem*{problem*}{Problem}

\newcommand*{\ZZ}{\mathbb{Z}}

\newcommand*{\RR}{\mathbb{R}}
\let\R\RR

\newcommand*{\TT}{\mathbb{T}}
\newcommand*{\Sp}{\mathbb{S}}

\newcommand\restr[2]{{
  \left.\kern-\nulldelimiterspace 
  #1 
  \right|_{#2} 
}}

\newcommand{\cdist}{\mathcal{H}} 
\newcommand{\Ldist}{\mathfrak{L}} 

\newcommand{\T}{{\mathsf T}}
\newcommand{\cT}{\T^{\ast}}
\newcommand*{\dd}{\mathrm{d}}

\newcommand*{\contr}[1]{\iota_{#1}}
\newcommand*{\liedv}[1]{\mathcal{L}_{#1}}
\newcommand*{\Reeb}{\mathcal{R}}

\DeclareMathOperator{\Image}{Im}
\let\Im\Image

\newcommand{\norm}[1]{\left\lvert\left\lvert #1 \right\rvert\right\rvert}

\let\oldemph\emph

\usepackage[style=ext-numeric,citestyle=numeric-comp,sorting=nyt,sortcites, backend=biber, giveninits=true, maxnames=99]{biblatex}
\bibliography{biblio}
\DeclareFieldFormat[article,periodical]{volume}{\mkbibbold{#1}}
\DeclareFieldFormat[article,periodical]{number}{\mkbibparens{#1}}
\renewbibmacro{in:}{%
  \ifentrytype{article}
    {}
    {\printtext{\bibstring{in}\intitlepunct}}}

\DeclareFieldFormat{issuedate}{#1}
\renewcommand{\jourvoldelim}{\addcomma\space}

\renewbibmacro*{journal+issuetitle}{%
  \usebibmacro{journal}%
  \setunit*{\jourvoldelim}%
  \iffieldundef{series}
    {}
    {\setunit*{\jourserdelim}%
     \printfield{series}%
     \setunit{\servoldelim}}%
  \usebibmacro{volume+number+eid}%
  \setunit{\addcolon\space}%
  \usebibmacro{issue}%
  \setunit{\bibpagespunct}%
  \printfield{pages}%
  \setunit{\space}%
  \usebibmacro{issue+date}%
  \newunit}

\renewbibmacro*{note+pages}{%
  \printfield{note}%
  \newunit}

\DeclareFieldFormat[article,periodical]{date}{\mkbibparens{#1}}

\AtEveryBibitem{\clearfield{month}}
\AtEveryBibitem{\clearfield{day}}

\AtEveryBibitem{
  \ifentrytype{online}{}{
    \ifentrytype{thesis}{}{
      \ifentrytype{unpublished}{}{ 
        \clearfield{url}
      }
    }
  }
  \ifentrytype{unpublished}{}{ 
  \clearfield{note}
  }
  \clearfield{language}
}

\DeclareSourcemap{
  \maps[datatype=bibtex]{
    \map[overwrite=true]{
      \step[fieldset=language, null]
      \step[fieldset=address, null]
     \step[fieldset=pagetotal, null]
    }
  }
}

\DeclareSourcemap{
    \maps[datatype=bibtex]{
      \map{
        \step[fieldsource=eprint,final]
        \step[fieldset=url,null]
        \step[fieldset=doi,null]
      }  
    }
  }

\AtEveryBibitem{\clearfield{pubstate}} 
\AtEveryBibitem{\clearfield{eprintclass}} 
\AtEveryBibitem{\clearfield{version}} 
\usepackage[unicode=true, pdfusetitle, colorlinks=true, allcolors=blue, bookmarks=false]{hyperref}
\usepackage[hyphenbreaks]{breakurl}
\usepackage{xurl}
\hypersetup{breaklinks=true}

\date{\today} 
\mathtoolsset{showonlyrefs}

\allowdisplaybreaks 

\begin{document}

\title[Liouville--Arnold theorem for homogeneous and contact Hamiltonian systems]{Liouville\,--Arnold theorem for homogeneous symplectic and contact Hamiltonian systems}

\author{Leonardo Colombo}
\address{Centre for Automation and Robotics, Spanish National Research Council (CSIC), Arganda del Rey, Madrid, Spain.}
\email{\mailto{leonardo.colombo@car.upm-csic.es}}

\author{Manuel de Le\'on}
\address{Institute of Mathematical Sciences, Spanish National Research Council (CSIC), Madrid, Spain; and Royal Academy of Exact, Physical and Natural Sciences, Madrid, Spain.}
\email{\mailto{mdeleon@icmat.es}}

\author{Manuel Lainz}
\address{Department of Quantitative Methods, CUNEF University, Madrid, Spain.}
\email{\mailto{manuel.lainz@cunef.edu}}

\author{Asier L\'opez-Gord\'on}
\address{Institute of Mathematics of the Polish Academy of Sciences (IM PAN), Warsaw, Poland.}
\email{\mailto{alopez-gordon@impan.pl}}

\thanks{The authors wish to express their gratitude to the referees for their constructive feedback.
They are also thankful to Robert Cardona, Guillermo Gallego and Eva Miranda for some fruitful discussions and valuable comments along the writing process of the paper. They acknowledge financial support from Grants PID2022-137909NB-C2 and RED2022-134301-T, funded by MCIN/AEI/ 10.13039/501100011033. Manuel de León, Manuel Lainz and Asier López-Gordón also received support from the Grant CEX2019-000904-S funded by MCIN/AEI/ 10.13039/501100011033.}

\date{\today}







\vspace{10pt}

\subjclass[2020]{Primary: 37J35, 37J55, 70H06; Secondary: 53D10, 53E50, 53Z05}

\keywords{integrable system, contact manifold, symplectic manifold, Hamiltonian system}

\begin{abstract}
A Hamiltonian system is completely integrable (in the sense of Liouville) if there exist as many independent integrals of motion in involution as the dimension of the configuration space. Under certain regularity conditions, Liouville--Arnold theorem states that the invariant geometric structure associated with Liouville integrability is a fibration by Lagrangian tori (or, more generally, Abelian Lie groups), on which the motion is linear.  In this paper, a Liouville--Arnold theorem for contact Hamiltonian systems is proven.  More specifically, it is shown that, given a $(2n+1)$-dimensional completely integrable contact system, one can construct a foliation by $(n+1)$-dimensional Abelian Lie groups and induce action-angle coordinates in which the equations of motion are linearized. One important novelty with respect to previous attempts is that the foliation consists of $(n+1)$-dimensional coisotropic submanifolds given by the preimages of rays by the functions in involution. In order to prove the theorem, we first develop a version of Liouville--Arnold theorem for homogeneous functions on exact symplectic manifolds (which is of independent interest), and then apply the symplectization to obtain the contact case.

\medskip


\end{abstract}

%
%
%
%
%


\dedicatory{Dedicated to the memory of Prof.~Miguel Carlos Muñoz-Lecanda.}

\maketitle

\section{Introduction}\label{sec_intro}

A large class of physical systems can be described through Hamiltonian systems \cite{Abraham2008,Arnold1978,Libermann1987,Godbillon1969,Goldstein1980,Wiggins2003}. In such systems, the dynamics is completely determined by the partial derivatives of the Hamiltonian function. The equations of motion are called Hamilton equations, a system of first-order differential equations. A Hamiltonian system is integrable if there exist explicit solutions to Hamilton’s equations of motion. Liouville--Arnold theorem \cite{Arnold1978,Bolsinov2004,Wiggins2003} asserts that, provided that there exist some functionally independent functions called first integrals, which are in involution with respect to a Poisson bracket, the equations of motion of a Hamiltonian system can be integrated by quadratures. That is, Hamilton equations can be solved by a finite number of algebraic operations and the calculation of integrals of known functions. More formally, given a symplectic manifold of dimension $2n$, a Hamiltonian system is (Liouville) completely integrable if one can find $n$ independent functions that pairwise Poisson-commute and are independent on a dense open subset. In this situation, this open subset admits a Lagrangian foliation and the solutions of the dynamics live on the leaves of the foliation. This notion may be extended, in a natural way, for the more general case when the phase space is a Poisson \cite{Miranda2022, Laurent-Gengoux2010} or Dirac manifold \cite{Zung2012}, not necessarily symplectic.

In recent years, there has been an increasing interest in the study of contact Hamiltonian systems (see \cite{deLeon2019,deLeon2019a,Grabowska2022,Grabowska2023a,Bravetti2017a,Montaldi2023} and references therein). Contact geometry has been used in the last years to describe dissipative mechanical systems from a Hamiltonian and Lagrangian point of view, as well as systems in thermodynamics \cite{Bravetti2019} and quantum mechanics \cite{Ciaglia2018}, among many others. Nevertheless, there is a lack in the literature regarding complete integrability of contact Hamiltonian systems as in the symplectic, Poisson, or Dirac scenarios.  In \cite{MirandaGalceran2003,MirandaGalceran2005} (see also \cite{Miranda2014}) a first result on action-angle coordinates for integrable systems with non-degenerate singularities on contact manifolds was given. In \cite{A.E2023}, the authors study the so-called particular integrals in autonomous and non-autonomous contact Hamiltonian systems, which allow to integrate the equations of motion by quadratures.

A notion of a completely integrable contact manifold has been given in \cite{Khesin2010}. In this reference, a contact manifold is said to be integrable if it is a flag of two foliations, Legendrian and co-Legendrian, and a holonomy-invariant transverse measure of the former in the latter. Additionally, an equivalence to previously known definitions of contact integrability \cite{Banyaga1993,Banyaga1999,Lerman2003, Libermann1991,Pang1990, Geiges2023, Jovanovic2015} is established. Moreover, in \cite{Khesin2010} the authors show that contact completely integrable systems are solvable in quadratures. However, these definitions are too restrictive for their application to dynamical systems with dissipation (see Section~\ref{sec:other_definitions}). Hence, we introduce a new definition of completely integrable contact system (see Definition~\ref{def:integrable_systems}).

Liouville--Arnold theorem \cite{Arnold1978,Audin2004,Fiorani2003,Bolsinov2004,Wiggins2003} states that on a $2n$-dimensional symplectic manifold with $n$ independent functions in involution $f_\alpha$, there exists a fibration by invariant Lagrangian tori. Moreover, on a neibourhood of each torus there exist the so-called action-angle coordinates. The action coordinates are related to the functions in involution $f_\alpha$ and specify a particular torus in the foliation. The angle coordinates are constructed using the flow of their Hamiltonian vector fields $X_{f_\alpha}$ and specify the position on each torus. Furthermore, the equations of motion are linear in these coordinates.

The generalization to contact Hamiltonian systems, however, is not straightforward. Assume that we are given a $(2n+1)$-dimensional contact manifold $M$ with $k$ functions $f_{\alpha}$ in involution with respect to the Jacobi bracket. If we naively try to construct the foliation by taking the level sets of $f_\alpha$ we run into problems. First of all, we have an issue with the dimensions of the leaves. Indeed, if $k=n$, the leaves are $(n+1)$-dimensional, but we only have $n$ Hamiltonian vector fields, and thus we are unable to construct angle coordinates. If $k = n + 1$, the leaves are $n$-dimensional but we have $n+1$ Hamiltonian vector fields and, hence, $n+1$ \enquote{angle} variables. Moreover, the level sets of the functions are, in general, not invariant by the flow of the Hamiltonian vector fields. Both problems can be fixed by taking $k=n+1$ but considering preimages of rays $M_{\ray{\lambda}} = \{x \in M \mid f_\alpha(x) = r \lambda_\alpha, \text{ for some $r>0$}\}$ instead of level sets. 

The main result of this paper is a Liouville--Arnold theorem for contact Hamiltonian systems (Theorem~\ref{thm:main_theorem}). By using a symplectization of a $(2n+1)$-dimensional completely integrable contact system, we show that it is possible to define a foliation whose leaves are invariant by the Hamiltonian flows of the functions in involution and diffeomorphic to $(n+1)$-dimensional Abelian Lie groups. In a neighbourhood of each of the leaves, one can introduce canonical coordinates in which the equations of motion are linearized. Furthermore, we have proven a Liouville--Arnold theorem for homogeneous functions on exact symplectic manifolds (Theorem~\ref{theorem:LA_exact_symp}). This result is stronger than the Liouville--Arnold theorem for non-compact manifolds \cite{Fiorani2003a}, since the diffeomorphism obtained not only preserves the symplectic form but also preserves its symplectic potential

The paper is structured as follows. Section~\ref{sec_preliminaries} introduces Jacobi and contact manifolds, in addition to Hamiltonian dynamics on contact manifolds. In Section~\ref{sec:theorem_exact_symp}, we prove a Liouville--Arnold theorem for homogeneous functions in involution on an exact symplectic manifold. This is
done in order to guarantee the homogeneous character on the action-angle coordinates that will be required when symplectizing a completely integrable contact system. Section~\ref{sec:main_theorem} contains the proof of the Liouville--Arnold theorem for contact Hamiltonian systems and the construction of action-angle coordinates. Essentially, the theorem is proven by considering the symplectization of the contact manifold, applying the theorem for exact symplectic manifolds, and then projecting the action-angle coordinates to the contact manifold. In Section~\ref{sec:coisotropic}, the properties of coisotropic manifolds are exploited to produce alternative characterizations of completely integrable contact systems. In Section~\ref{sec:other_definitions}, we provide a review of the literature on contact integrable systems, comparing the results and definitions of other authors with ours. 
Finally, to show an application of the theory we have developed, an illustrative example is studied in Section~\ref{sec:example}. 

\textbf{Notation and conventions.} From now on, all the manifolds and mappings are assumed to be smooth {and connected}. Submersions are assumed to be surjective. Fiber bundles are assumed to be locally trivial. Sum over crossed repeated indices is understood.

Given a manifold $M$, the $\Cinfty(M)$-modules of vector fields and of $p$-forms on $M$ will be denoted by $\X(M)$ and by $\Omega^p(M)$, respectively. For $X\in\X(M)$ and $\alpha\in \Omega^p(M)$,
the exterior derivative of $\alpha$ is denoted by $\dd \alpha$, the contraction of $\alpha$ with $X$ is denoted by $\contr{X}\alpha$. The Lie derivative of a tensor field $A$ on $M$ with respect to $X$ is denoted by $\liedv{X} A$. Given a collection of vector fields $X_1, \ldots, X_n \in \X(M)$, the distribution spanned by them is denoted by $\langle  X_1, \ldots, X_n  \rangle$, and analogously for codistributions spanned by one-forms. The Lie bracket on $\X(M)$ is denoted by $[\cdot, \cdot]$.

The tangent and cotangent bundles of a manifold $M$ will be denoted by $\tau_M \colon \T M\to M$ and $\pi_M \colon \cT M\to M$, respectively. Given two manifolds $M$ and $N$, and a map $F\colon M \to N$, the tangent map of $F$ will be denoted by $\T F \colon \T M \to \T N$. If $F$ is a diffeomorphism, the pushforward of $X\in \X(M)$ by $F$ will be denoted by $F_\ast X$, that is, $(F_\ast X)_x = \T F_{F^{-1}(x)} \left(X_{F^{-1}(x)}\right)$.
The positive real numbers will be denoted by $\R_+$. The $n$-dimensional torus will be denoted by $\TT^n$, {defined as $\TT^n = \RR^n/\ZZ^n$}.
Latin indices $i,j,k$ will take values in $\{1, \ldots, n\}$, whilst Greek indices $\alpha, \beta, \gamma$ will take values in $\{0, 1, \ldots, n\}$.

\section{Jacobi and contact manifolds} \label{sec_preliminaries}

\subsection{Jacobi manifolds}

{
Let $M$ be a manifold.
Consider a map $\{\cdot, \cdot\}\colon \Cinfty(M)\times \Cinfty(M)\to \Cinfty(M)$ given by
\begin{equation}\label{eq:Jacobi_bracket}
    \{f, g\} = \Lambda(\dd f\wedge \dd g) + f E(g) - g E(f)\, ,
\end{equation}
where $\Lambda\in\Gamma(\bigwedge^2 \X(M))$ is a bivector field and $E\in \X(M)$ a vector field. Lichnerowicz \cite{Lichnerowicz1977, Lichnerowicz1977a, Lichnerowicz1978} proved that $\{\cdot, \cdot\}$ is a Lie bracket if and only if 
\begin{equation}\label{eq:Jacobi_structure}
    [\Lambda, \Lambda]_{\mathrm{SN}} = 2 E \wedge \Lambda\, , \quad [E, \Lambda]_{\mathrm{SN}} = 0 \, ,
\end{equation}
where $[\cdot, \cdot]_{\mathrm{SN}}$ denotes the Schouten--Nijenhuis bracket (see \cite{Grabowski2013, Michor1987} and references therein for more details on this bracket).
If these conditions hold, then $\{\cdot, \cdot\}$ is called a \emph{Jacobi bracket on $M$}, the pair $(\Lambda, E)$ is called a \emph{Jacobi structure on $M$}, and the triple $(M, \Lambda, E)$ is called a \emph{Jacobi manifold}. By construction, the Jacobi bracket is a Lie bracket, that is, it is bilinear, skew-symmetric and satisfies the Jacobi identity. Moreover, it satisfies the so-called \emph{weak Leibniz rule}, namely, 
\begin{equation}\label{eq:weak_Leibniz_rule}
    \{f, gh\} = \{f, g\} h + \{f, h\} g + g h E(f)\, ,
\end{equation} 
for any $f, g, h\in \Cinfty(M)$, where $gh$ denotes the pointwise multiplication.

It is possible to consider an equivalent, more algebraic, definition Jacobi structures on manifolds.
Since the Jacobi bracket $\{\cdot, \cdot\}$ is a Lie bracket and satisfies the weak Leibniz rule, it defines a structure of local Lie algebra in the sense of Kirillov on $\Cinfty(M)$ (see \cite{Kirillov1976,Dazord1991,G.L1984}). Conversely, given a local Lie algebra on $\Cinfty(M)$, there exists a Jacobi structure on $M$ such that the Jacobi bracket coincides with the algebra bracket. 
}

The bivector $\JacLambda$ induces a $\Cinfty(M)$-module morphism $\lsharp\colon\Omega^1(M)\to\X(M)$ given by
\begin{equation}
   \lsharp(\alpha) = \JacLambda(\alpha,\cdot)\, .
\end{equation}


The \emph{characteristic distribution} of a Jacobi manifold $(M, \JacLambda, E)$ is given by $C=\lsharp(\cT M)\oplus \langle E \rangle$. 
{The so-called \emph{Hamiltonian vector field} $X_f$ of $f\in \Cinfty(M)$ is given by
$$X_f = \lsharp(\dd f) + f E\, ,$$
or equivalently,
\begin{equation}\label{eq:Hamiltonian_vf_Jacobi_bracket}
    X_f(g) = \{f,g\} + g E(f)\, , \quad \forall\, g\in \Cinfty(M)\, .
\end{equation}
}

\begin{definition}
    Let $(M, \JacLambda, E)$ be a Jacobi manifold with Jacobi bracket $\{\cdot, \cdot\}$.
    A collection of functions $f_1, \ldots, f_k\in \Cinfty(M)$ are said to be \emph{in involution} if
    $\{f_i, f_j\}=0$ for any $i,j=1, \ldots, k$.
\end{definition}

The notions of coisotropic and Lagrangian submanifolds have been deeply studied in symplectic and Poisson geometry \cite{Libermann1987,Vaisman1994}. The analogous notions for Jacobi manifolds were first studied by Ibáñez, de León, Marrero and Martín de Diego \cite{Ibanez1997} (see also \cite{Le2018}).

\begin{definition}
    Given a Jacobi manifold $(M, \JacLambda, E)$ with characteristic distribution $C$, we define the Jacobi orthogonal complement of a distribution $D \subseteq \T M$ as 
    \begin{equation}
        D^{\perp_\JacLambda} = \lsharp( D^\circ )\, ,
    \end{equation}
    where $D^\circ$ is the annihilator of $D$, that is, the codistribution given by
    \begin{equation}
        D^\circ_x = \{\alpha \in \cT_x M \mid \alpha(v) = 0\ \forall\, v\in D_x\}\, ,
    \end{equation}
    for each $x\in M$.

    We say that a submanifold $N \hookrightarrow M$ is 
    \begin{itemize}
        \item \emph{Coisotropic} if $\T N^{\perp_\JacLambda} \subseteq \T N$.
        \item \emph{Legendrian} if $\T N^{\perp_\JacLambda} = \T N \cap C$.
    \end{itemize}
\end{definition}

\subsection{Contact manifolds}

Let $M$ be a $(2 n+1)$-dimensional manifold. A \emph{contact distribution} on $M$ is a corank-one maximally non-integrable distribution $\cdist$ on $M$. Being maximally non-integrable means that the bilinear map
$$\nu_\cdist\colon \cdist \times_M \cdist \to \T M/\cdist\, , \quad \nu_\cdist\Big([X,Y]\Big) = \gamma(X, Y)$$
is non-degenerate, where $[\cdot, \cdot]$ denotes the restriction to sections of $\cdist$ of the canonical Lie bracket on $\X(M)$, and $\gamma\colon \T M \to \T M/\cdist$ is the canonical projection.
The pair $(M,\cdist)$ is  called a \emph{contact manifold}. A \emph{contact form} on $M$ is a one-form $\eta\in\Omega^1(M)$ such that $\ker\eta$ is a contact distribution on $M$. If such one form exists globally, $(M,\eta)$ is called a \emph{co-oriented contact manifold} \cite{Geiges2008}. Henceforth, all the contact manifolds considered will be co-oriented. Hence, the following definition will be employed.

\begin{definition}
    A \emph{contact manifold} is a pair $(M, \eta)$, where $M$ is a $(2 n+1)$-dimensional manifold and $\eta$ is a one-form on $M$ such that $\eta \wedge(\mathrm{d} \eta)^{n}$ is a volume form.
The one-form $\eta$ is called a \emph{contact form}. 
\end{definition}

The contact form $\eta$ on $M$ defines the vector bundle isomorphism $\flat : \T M \to \cT M$
\begin{equation*}
\begin{aligned}
  \flat \colon
  v & \mapsto \contr{v} \dd \eta + \eta(v) \eta\, ,
\end{aligned}
\end{equation*}
which induces an isomorphism of $\Cinfty(M)$-modules $\flat\colon \X(M) \to \Omega^1(M)$.
There exists a unique vector field $\Reeb$ on $(M,\eta)$, called the \emph{Reeb vector field}, such that $\flat(\Reeb) = \eta$, that is,
\begin{equation*}
   \contr{\Reeb} \dd \eta = 0\, ,\quad 
   \contr{\Reeb} \eta = 1\, .
\end{equation*}
Given a function $f\in \Cinfty(M)$,  its \emph{Hamiltonian vector field} $X_f$ is given by
\begin{equation}
  \flat(X_f) = \dd f - \left(\Reeb (f) + f  \right) \eta\, ,
  \label{Hamiltonian_vector_field}
\end{equation}
or, equivalently, 
 \begin{subequations}
 \begin{flalign}
   &\eta(X_f) = -f\, ,\quad \liedv{X_f} \eta = - \Reeb(f) \eta\, . \label{lieder_eta}
 \end{flalign}
 \label{Eqs_Hamiltonian_vector_field}
 \end{subequations}



Around each point of a $(2n+1)$-dimensional contact manifold $(M, \eta)$ there exist local coordinates, called \emph{Darboux coordinates}, $(q^1, \ldots, q^n, p_1, \ldots, p_n, z)$ such that $\eta = \dd z - p_i \dd q^i$. In these coordinates, the Reeb vector field is written as $\Reeb = \partial/\partial z$, and the Hamiltonian vector field of $f\in \Cinfty(M)$ reads
\begin{equation}
    X_f = \frac{\partial f}{\partial p_i} \frac{\partial}{\partial q^i} - \left(\frac{\partial f}{\partial q^i}+p_i \frac{\partial f}{\partial z}\right) \frac{\partial }{\partial p_i} + \left(p_i \frac{\partial f}{\partial p_i} - f\right) \frac{\partial }{\partial z}\, .
\end{equation}

\begin{definition}
    Let $(M_1, \eta_1)$ and $(M_2, \eta_2)$ be contact manifolds. A diffeomorphism $\Phi\colon M_1 \to M_2$ is called a \emph{{conformal contactomorphism}}\footnote{{In the literature, sometimes conformal contactomorphisms are simply called contactomorphisms.}} if $\Phi^\ast\eta_2 = \sigma \eta_1$ for some nowhere-vanishing function $\sigma$ on $M_1$ called the \emph{conformal factor}. A \emph{strict contactomorphism}
    is a diffeomorphism $\Phi\colon M_1 \to M_2$ which preserves the contact forms, namely, $\Phi^\ast \eta_2 = \eta_1$.

    Let $(M, \eta)$ be a contact manifold.
    A vector field $X\in \X(M)$ is called an \emph{infinitesimal {conformal contactomorphism}} if its flow is a one-parameter group of {conformal contactomorphisms}. Equivalently, $\liedv{X} \eta = a \eta$ for some function $a$ on $M$ called the \emph{(infinitesimal) conformal factor}. An \emph{infinitesimal strict contactomorphism} is a vector field $X\in \X(M)$ whose flow is made up of strict  contactomorphisms, that is, $\liedv{X}\eta = 0$.
    
    Note that an (infinitesimal) strict contactomorphism is an (infinitesimal) {conformal contactomorphism} with (infinitesimal) conformal factor ($a\equiv0$) $\sigma\equiv1$. 
\end{definition}


    Let $(M,\eta)$ be a contact manifold. Then, $(M,\JacLambda,E)$ is a Jacobi manifold, where
    $$ \JacLambda(\alpha,\beta) = -\dd\eta(\flat^{-1}\alpha,\flat^{-1}\beta)\ ,\quad E = -\Reeb\,. $$
 Hence, the associated Jacobi bracket $\left\{\cdot, \cdot  \right\}$ is given by 
 \begin{equation*}
   \left\{f, g  \right\} 
   = \JacLambda(\dd f, \dd g)  + f E(g) - g E(f)
    = -\dd \eta (\flat^{-1} \dd f, \flat^{-1} \dd g)
    - f \Reeb(g) +g\Reeb(f),
 \end{equation*}
 for any pair of smooth functions $f$ and $g$ on $M$. Moreover, the characteristic distribution is the complete tangent space, namely, $C=\T M$.



{Given a contact manifold $(M, \eta)$, the morphism $\lsharp\colon \Omega^1(M) \to \X(M)$ defined by its Jacobi structure is given by $\lsharp(\alpha) = \flat^{-1} (\alpha) -\alpha(\Reeb) \Reeb$. Thus, the Hamiltonian vector field of $f\in \Cinfty(M)$ can be written as 
\begin{equation}
    X_f = \lsharp (\dd f) - f \Reeb\, .
\end{equation}
In particular, for each $f\in \Cinfty(M)$, its Hamiltonian vector field  with respect to $\eta$ coincides with its Hamiltonian vector field with respect to the associated Jacobi structure $(\Lambda, E)$.}


\begin{proposition}\label{prop:hamiltonian_bijection}
    Given a contact manifold $(M,\eta)$, the map $f \mapsto X_f$ is a Lie algebra anti-isomorphism between the set of smooth functions with the Jacobi bracket and the set of infinitesimal {conformal contactomorphisms} with the Lie bracket, namely,
    \begin{equation}
         \{f,g\} = - \eta\left([X_f, X_g]\right)\, ,
    \end{equation}
    for any $f,g \in \Cinfty(M)$. Its inverse is given by $X \mapsto - \eta(X)$. 
    
    Furthermore, $X_f$ is an infinitesimal strict contactomorphism if and only if $\Reeb(f) = 0$.
\end{proposition}

\begin{proof}
    Let $f\in \Cinfty(M)$. Since $\liedv{X_f} \eta = - \Reeb(f) \eta$, $X_f$ is an infinitesimal {conformal contactomorphism}, and it is an infinitesimal strict contactomorphism if and only if $\Reeb (f) = 0$. Conversely, for any $X\in \X(M)$, the contraction $f\coloneqq\eta(X)$ is a smooth function on $M$. Moreover, if $\liedv{X} \eta = \gamma \eta$, then
    $$\liedv{\Reeb} f = - \liedv{\Reeb}\contr{X} \eta = -\contr{[\Reeb, X]}\eta - \contr{X} \liedv{\Reeb} \eta = \contr{[X, \Reeb]} \eta
    =\liedv{X} \contr{\Reeb} \eta - \contr{\Reeb} \liedv{X} \eta = 
    - \contr{\Reeb} \left(\gamma \eta\right) = -\gamma\, .$$
     This proves that the map $f\mapsto X_f$ is a bijection.

    We shall now show that $X \to - \contr{X} \eta$ is an anti-isomorphism. Since it is a bijection (with inverse $f \mapsto X_f$) it is sufficient to show that it is an anti-homomorphism.
    We have that
    \begin{equation}
        \begin{aligned}
            \dd \eta(X_f,X_g) 
            & = X_f(\eta(X_g)) -  X_g (\eta(X_f)) -\contr{\lieBr{X_f,X_g}} \eta \\ &= 
            - X_f(g)  + X_g(f)   -\contr{\lieBr{X_f,X_g}} \eta \\ &=
            - \lsharp(\dd f)(g) + f \Reeb (g) + \lsharp(\dd g)(f) - g \Reeb(f) -\contr{\lieBr{X_f,X_g}} \eta \\ &= 
            -2 \JacLambda(\dd f, \dd g)+
             f\Reeb(g) - g \Reeb(f) -\contr{\lieBr{X_f,X_g}} \eta\, ,
        \end{aligned}
    \end{equation}
    since $\lsharp(\dd f)(g) = -\lsharp(\dd g)(f) = \JacLambda(\dd f, \dd g)$. Taking into account that
    \begin{equation}
        \JacLambda(\dd f, \dd g) = -\dd \eta (X_f, X_g)\, ,
    \end{equation}
    one can write
    \begin{equation}
        \JacLambda (\dd f, \dd g) = f\Reeb(g) - g \Reeb(f) -\contr{\lieBr{X_f,X_g}} \eta\, ,
    \end{equation}
    and then
    \begin{equation}
        \jacBr{f,g} 
        = \JacLambda(\dd f, \dd g) -f \Reeb (g) + g \Reeb(f)
        = - \contr{\lieBr{X_f,X_g}} \eta\, .
    \end{equation}
\end{proof}

\subsection{Contact Hamiltonian systems}
\begin{definition}
    A \emph{contact Hamiltonian system} is a triple $(M,\eta, h)$, where $(M, \eta)$ is a contact manifold and $h\colon M \to \R$ is a function called the \emph{Hamiltonian function}.
\end{definition}

 The dynamics of $(M,\eta, h)$ is given by the integral curves of the Hamiltonian vector field $X_h$ of $h$ with respect to $\eta$. In Darboux coordinates, these integral curves $c\colon I\subseteq \RR \to M,\, c(t)=(q^i(t), p_i(t), z(t))$ are determined by the \emph{contact Hamilton equations}:
 \begin{equation}
    \frac{\dd q^i}{\dd t} (t) = \frac{\partial h}{\partial p_i} \circ c(t)\, , \quad
    \frac{\dd p_i}{\dd t} (t) = -\frac{\partial h}{\partial q^i} \circ c(t) - p_i(t) \frac{\partial h}{\partial z} \circ c(t)\, , \quad
    \frac{\dd z}{\dd t} (t) =  p_i(t) \frac{\partial p_i}{\partial z} \circ c(t) - h\circ c(t)\, .
 \end{equation}

\begin{definition}
    Consider a contact Hamiltonian system $(M, \eta, h)$, and let $X_h$ denote the Hamiltonian vector field of $h$ with respect to $\eta$. A \emph{dissipated quantity (with respect to $(M, \eta, h)$)} is a function $f\in \Cinfty(M)$ such that $X_h(f)=-\Reeb(h) f$. A \emph{conserved quantity (with respect to $(M, \eta, h)$)} is a function $f$ on $M$ such that $X_h(f)=0$.
\end{definition}

Let $\{\cdot, \cdot\}$ denote the Jacobi bracket defined by $\eta$.
Observe that a function $f$ on $M$ is a dissipated quantity if and only if $\{f, h\} = 0$. In particular, $h$ is a dissipated quantity.
If $f_1$ and $f_2$ are dissipated quantities, then $f_1/f_2$ is a conserved quantity on the open subset $U=M\setminus f_2^{-1}(0)\subseteq M$.

\section{Liouville--Arnold theorem for exact symplectic manifolds}\label{sec:theorem_exact_symp}

An \emph{exact symplectic manifold} is a pair $(M, \theta)$, where $M$ is a manifold and $\theta$ a one-form on $M$ called the \emph{symplectic potential} such that $\omega= -\dd \theta$ is a symplectic form on $M$. The \emph{Liouville vector field} $\Delta$ of $(M, \theta)$ is given by 
\begin{equation}\label{eq:Liouville_vf}
    \contr{\Delta} \omega = -\theta \, .
\end{equation}
Given an integer $k$, a tensor field $A$ on $M$ is called \emph{homogeneous of degree $k$} if $\liedv{\Delta} A = k A$.

 Given a function $f$ on $(M,\theta)$, its \emph{(symplectic) Hamiltonian vector field} $X_f$ is given by
 \begin{equation}
     \contr{X_f} \omega = \dd f\, .
 \end{equation}
 If $f$ is homogeneous of degree 1, then
 \begin{equation}
     \theta(X_f) = f\, .
 \end{equation}

{ An exact symplectic manifold is canonically equipped with a Liouville vector field uniquely determined by \eqref{eq:Liouville_vf}. Conversely, if a symplectic manifold $(M, \omega)$ is endowed with a distinguished vector field $\Delta$ --the so-called weight vector field-- such that $\liedv{\Delta}\omega = \omega$, then $\theta = - \contr{\Delta} \omega$ is a symplectic potential for $\omega$. This idea has been studied in the context of Hamiltonian systems with scaling symmetries \cite{B.J.S2023,B.G.M+2024}. For more details about weight vector fields and homogeneous objects, see \cite{G.G.R2022,G.G2024} and references therein.}

 Hereinafter, it will be assumed that the exact symplectic manifolds have no \emph{singular points}, that is, those in which $\Delta$ (or, equivalently, $\theta$) vanish. Notice that those points form a closed subset which is contained on a submanifold of dimension at most half of the dimension of the original manifold. In the case that $M= \cT  Q$ is a cotangent bundle with $\theta_Q$ the canonical one-form and $\Delta_Q$ the canonical Liouville vector field, the zero section would be removed. Recall that, if $Q$ has local coordinates $(q^i)$ and $(q^i, p_i)$ are the induced bundle coordinates on $\cT Q$, then $\theta_Q = p_i \dd q^i$ and $\Delta_Q = p_i \partial/\partial p_i$.

Let $\{\cdot, \cdot\}_{\theta}$ denote the Poisson bracket defined by the symplectic structure $\omega=-\dd \theta$, that is,
\begin{equation}
    \{f, g\}_{\theta} = \omega(X_f, X_g)\, ,
\end{equation}
for each $f, g\in \Cinfty(M)$. Given the functions $f_i\colon M\to \R,\ i=1,\ldots, n$, define the map 
\begin{equation}
    F = \left(f_1, \ldots, f_n\right) \colon M \to \R^{n}\, ,
\end{equation}
and the level set $M_\lambda = F ^{-1}(\lambda)$, for each $\lambda\in \R^{n}$.

\begin{definition}
    Given two exact symplectic manifolds $(M_1, \theta_1)$ and $(M_2, \theta_2)$, a \emph{homogeneous symplectomorphism} is a diffeomorphism $\Phi\colon M_1 \to M_2$ such that $\Phi^\ast \theta_2 = \theta_1$. 
    
    Given a symplectic manifold $(M, \theta)$, an \emph{infinitesimal homogeneous symplectomorphism} is a vector field $X\in \X(M)$ whose flow is a one-parameter group of homogenous symplectomorphisms or, equivalently, $\liedv{X}\theta = 0$.
\end{definition}

Clearly, every (infinitesimal) homogeneous symplectomorphism is an (infinitesimal) symplectomorphism. Moreover, if $\Delta_1$ and $\Delta_2$ are the Liouville vector fields of $(M_1, \theta_1)$ and $(M_2, \theta_2)$, and $\Phi\colon M_1 \to M_2$ is a homogeneous symplectomorphism, then $\Phi_\ast \Delta_1 = \Delta_2$.

\begin{proposition}\label{proposition:homogeneous_Hamiltonian}
    Let $(M, \theta)$ be an exact symplectic manifold.
    Given a vector field $Y\in \X(M)$, the following statements are equivalent:
    \begin{enumerate}
        \item \label{item:homogeneous_symp} $Y$ is an infinitesimal homogeneous symplectomorphism,
        \item \label{item:commute_Delta} $Y$ is an infinitesimal symplectomorphism and commutes with the Liouville vector field $\Delta$,
        \item \label{item:Hamiltonian} $Y$ is the Hamiltonian vector field of $f= \theta(Y)$ and $f$ is a homogeneous function of degree $1$.
    \end{enumerate}
\end{proposition}

\begin{proof}
    The first step is to show the equivalence between statements \ref{item:homogeneous_symp} and \ref{item:commute_Delta}.
    Suppose that $Y\in \X(M)$ is an infinitesimal symplectomorphism. Then,
    \begin{equation}
        \contr{[Y, \Delta]} \omega = \liedv{Y} \contr{\Delta} \omega - \contr{\Delta} \liedv{Y} \omega = - \liedv{Y} \theta\, ,
    \end{equation}
    which implies that $Y$ is an infinitesimal homogeneous symplectomorphism if and only if it commutes with $\Delta$.

    The next step is to show that \ref{item:homogeneous_symp} implies \ref{item:Hamiltonian}.
    Suppose that $Y\in \X(M)$ is an infinitesimal homogeneous symplectomorphism. Then, 
    \begin{equation}
        0 = \liedv{Y} \theta = \contr{Y} \dd \theta + \dd \contr{Y} \theta\, ,
    \end{equation} 
    which implies that 
    \begin{equation}
        \contr{Y} \omega = \dd f\, ,
    \end{equation}
    for $f = \contr{Y} \theta$.

    
    The last step is to show that \ref{item:Hamiltonian} implies \ref{item:commute_Delta}. Suppose that $f\in \Cinfty(M)$ is a homogeneous function of degree $1$ and $X_f$ is its Hamiltonian vector field. Then, 
    \begin{equation}
        \contr{[\Delta, X_f]} \omega = \liedv{\Delta} \contr{X_f} \omega - \contr{X_f} \liedv{\Delta} \omega = \liedv{\Delta} \dd f - \dd f = 0\, . 
    \end{equation}
\end{proof}

On this context, we have a notion of integrable system:
\begin{definition}\label{def:integrable_systems_homogeneous}
    A \emph{homogeneous integrable system} is a triple $(M, \theta, F)$, where $(M, \theta)$ is an exact symplectic manifold and $F=(f_1, \ldots, f_n)\colon M\to \RR^{n}$ is a map such that the functions $f_1, \ldots, f_n$ are independent, in involution (i.e., $\jacBr{f_i,f_j}_\theta=0$) and homogeneous of degree $1$ on $M$. The functions $f_1, \ldots, f_n$ are called \emph{integrals}.
\end{definition}

If there are singularities, such as critical points of the integrals, on a zero-measure subset $S$ of $M$, it suffices to restrict to the dense open subset $M_{0} = M\setminus S$.

\begin{theorem}[Liouville--Arnold theorem for exact symplectic manifolds]
\label{theorem:LA_exact_symp}
Let $(M, \theta, F)$ be a $2n$-dimensional homogeneous integrable system with $F = {(f_i)}_{i=1}^n$. Given $\lambda\in \RR^n$, assume that $M_\lambda = F^{-1}(\lambda)$ is connected (if not, replace $M_\lambda$ by a connected component).
Let $U$ be an open neighbourhood of $M_\lambda$ such that:
\begin{enumerate}
    \item $f_1, \dots, f_n$ have no critical points in $U$,
    \item the Hamiltonian vector fields $X_{f_1}, \dotsc, X_{f_n}$ are complete,
    \item the submersion $F\colon U \to \R^n$ is a trivial bundle over a domain $V\subseteq \R^n$.
\end{enumerate}
Then, $U$ is diffeomorphic to $W=\TT^k \times \R^{n-k}\times V$, where $\TT^k$ denotes the $k$-dimensional torus. Furthermore, there is a chart $(\hat{U}\subseteq U; y^i, A_i)$ of $M$
such that:
\begin{enumerate}
    \item $A_i = M_i^j f_j$, where $M_i^j$ are homogeneous functions of degree $0$ depending only on $f_1, \ldots, f_n$,
    \item $(y^i, A_i)$ are canonical coordinates for $\theta$, i.e., $ \theta = A_i \dd y^i$,
    \item the Hamiltonian vector fields of the functions $f_i$ read 
    \begin{equation}
        X_{f_i} = N_i^j \frac{\partial}{\partial y^j}\, ,
    \end{equation}
    with $(N_i^j)$ the inverse matrix of $(M_i^j)$.
\end{enumerate}
In this chart, the integral curves of $X_{f_j}$ are given by
\begin{equation}
\begin{aligned}
    & \dot y^i = \Omega^i (A_1, \ldots, A_n)\, ,\\
   & \dot A_i = 0\,  ,
   \quad i= 1,\ldots, n\, ,
\end{aligned}
\end{equation}
where $\Omega^i\coloneqq N^i_j$.
\end{theorem}

The coordinates $(y^i)$ and $(A_i)$, for $i=1,\dotsc n$, are called \emph{angle coordinates} and \emph{action coordinates}, respectively.

The process for constructing the action-angle coordinates can be retrieved from the proof of the theorem.

\begin{remark}[Construction of action-angle coordinates]\label{remark:construction_coords_exact}
    In order to construct action-angle coordinates in a neighbourhood $U$ of $M_{\lambda}$, one has to carry out the following steps:
    \begin{enumerate}
        \item Fix a section $\chi$ of $F\colon U \to V$ such that $\chi^\ast \theta = 0$.
        \item Compute the flows $\phi_t^{X_{f_i}}$ of the Hamiltonian vector fields $X_{f_i}$.
        \item Let $\Phi\colon \R^{n}\times M \to M$ denote the action of $\R^{n}$ on $M$ defined by the flows, namely,
        $$\Phi(t_1, \ldots, t_n; x) = \phi_{t_1}^{X_{f_1}}\circ \cdots \circ\phi_{t_n}^{X_{f_n}}(x)\, .$$
        \item It is well-known that the isotropy subgroup $G_{\chi(\lambda)(\lambda)}=\{g\in \R^{n}\mid \Phi(g,\chi(\lambda))=\chi(\lambda)\}$, forms a lattice (i.e., a $\mathbb{Z}$-submodule of $\R^{n}$). Pick a $\mathbb{Z}$-basis $\{e_1,\ldots, e_m\}$, where $m$ is the rank of the isotropy subgroup.
        \item Complete it to a basis $\mathcal{B}=\{e_1, \ldots, e_m, e_{m+1}, \ldots, e_{n}\}$ of $\R^{n}$.
        \item Let $(M_i^j)$ denote the matrix of change from the basis $\{X_{f_1}(\chi(\lambda)), \ldots, X_{f_n}(\chi(\lambda))\}$ of $\T_{\chi(\lambda)} M_\lambda\simeq \RR^{n}$ to the basis $\{e_i\}$. The action coordinates are the functions $A_i = M_i^j f_j$.
        \item The angle coordinates $(y^i)$ of a point $x\in M$ are the solutions of the equation
        \begin{equation}
             x = \Phi\big(y^i e_i;\chi \circ F(x)\big)\, .
        \end{equation}
    \end{enumerate}
 \end{remark}

In order to prove Theorem~\ref{theorem:LA_exact_symp} some lemmas will be required.

\begin{lemma}\label{lemma:Arnold_tangent_fields}
    Let $M$ be an $n$-dimensional differentiable manifold and let $X_1, \ldots, X_n\in \X(M)$ be linearly independent vector fields. If these vector fields {pairwise commutate and are} complete, then $M$ is diffeomorphic to $\TT^k\times \R^{n-k}$ for some $k\leq n$, where $\TT^k$ denotes the $k$-dimensional torus.
\end{lemma}

See Lemmas 1.3 and 1.4 in reference \cite{Bolsinov2004} for the proof.



\begin{lemma}\label{lemma:linear_combination_functions}
    Let $(M, \theta, (f_i)_{i=1}^n)$ be a $2n$-dimensional homogeneous integrable system. Assume that their Hamiltonian vector fields $X_{f_i}$ are complete. Then, there exists functions $(g_i)_{i=1}^n$ functionally dependent on the $f_i$ such that $(M, \theta, (g_i)_{i=1}^n)$ is a homogeneous integrable system, and, moreover,
        \begin{enumerate}
        \item $X_{g_1}, \ldots, {X_{g_k}}$ are infinitesimal generators of $\Sp^1$-actions and their flows have period 1,
        \item $X_{g_{k+1}}, \ldots, X_{g_n}$ are infinitesimal generators of $\RR$-actions.
    \end{enumerate}
    Here $k$ is the dimension of the isotropy subgroup from the action generated by the Hamiltonian vector fields $X_{f_i}$. 

     These functions are given by $g_i = M^j_i f_j \in \Cinfty(M)$, where $M^j_i$ for $i, j\in {1, \ldots, n}$ are homogeneous functions of degree $0$ and they depend only on $f_1, \ldots, f_n$.
\end{lemma}

\begin{proof}
    {
    Note that the Hamiltonian vector fields $X_{f_1}, \ldots, X_{f_n}$ are tangent to each level set $M_\lambda=F^{-1}(\lambda)$ of $F=(f_1, \ldots, f_n)$, with $\lambda\in \RR^{n}$. Indeed, 
    $$X_{f_i} (f_j) = \{f_i, f_j\} = 0\, , \quad \forall\, i,j\in \{1, \ldots, n\}\, ,$$
    in other words, $X_{f_i}$ is tangent to the level sets of $f_j$.
    Consequently, the restriction of the vector fields $X_{f_1}, \ldots, X_{f_n}$ to $M_\lambda$ generate the Abelian Lie algebra
    \begin{equation}
        \mathfrak{g} = \operatorname{Lie} (\TT^k \times \RR^{n-1}) \cong
        \operatorname{Lie} (\Sp^1) \oplus \cdots\oplus  \operatorname{Lie} (\Sp^1) \oplus \operatorname{Lie} (\RR) \oplus \cdots \oplus \operatorname{Lie} (\RR) 
        \, , 
    \end{equation}
    where there are $k$ copies of $\operatorname{Lie} (\Sp^1)$ and $n-k$ copies of $\operatorname{Lie} (\RR)$. Hence, one can choose new generators given by linear combinations of $\restr{X_{f_i}}{M_\lambda}$ which are adapted to the direct sum. Furthermore, since $X_{f_i}$ are all homogeneous of degree $0$, we can make linear combinations that are also homogeneous of degree $0$.
    }
    More specifically, there exist $Y_1, \ldots, Y_n \in \X(M_\lambda)$ with $Y_i = M^j_i \restr{X_{f_j}}{M_\lambda}$, for some real parameters $M^j_i$, such that
    \begin{itemize}
        \item for $1\leq i \leq k$ the vector field $Y_i$ is the generator of the $i$-th copy of $\operatorname{Lie} (\Sp^1)$,
        \item for $k+1\leq i \leq n$ the vector field $Y_i$ generates the $(i-k)$-th copy of $\operatorname{Lie} (\RR)$,
        \item $[Y_i, \Delta]=0$ for all $i\in \{1, \ldots, n\}$. 
    \end{itemize}
    Without loss of generality, the generators of $\operatorname{Lie} (\Sp^1)$ can be assumed to have flows with period $1$. (If that is not the case it suffices to rescale $Y_1, \ldots, Y_k$.) Since the level sets $M_{\lambda}$ foliate $M$, these vector fields can be extended to the whole manifold, namely,
    \begin{equation}
        Y_i = M^j_i (f_1, \ldots, f_n) X_{f_j}\in \X(M)
    \end{equation}
    where $M^j_i$ are functions which depend only on $f_1, \ldots, f_n$. 

    Since the functions $f_i$ are homogeneous of degree $1$ their Hamiltonian vector fields are infinitesimal homogeneous symplectomorphisms. In other words, $\theta$ is $\mathfrak{g}$-invariant. Hence, $Y_1, \ldots, Y_n$ are infinitesimal homogeneous symplectomorphisms. By Proposition~\ref{proposition:homogeneous_Hamiltonian}, each vector field $Y_i=X_{g_i}$ is the Hamiltonian vector field of
    \begin{equation}
        g_i = \theta (Y_i) = M^j_i f_j\, ,
    \end{equation}
    which is a homogeneous function of degree $1$. Therefore, $M^j_i$ are homogeneous functions of degree $0$. The fact that 
    \begin{equation}
        X_{\{h_i, h_j\}_\theta} = - [X_{h_i}, X_{h_j}]\, ,
    \end{equation}
    for all $h_i, h_j\in \Cinfty(M)$, implies that the functions $g_1, \ldots, g_n$ are in involution. Finally, taking into account that the functions $M^j_i$ are fixed on each $M_\lambda$ and $Y_i$ are tangent to $M_{\lambda}$, the completeness of $Y_i$ follows from the completeness of $X_{f_i}$.  
\end{proof}

\begin{lemma}\label{lemma:global_section}
    Let $\pi \colon P\to M$ be a $G$-principal bundle over a connected and simply connected manifold. Suppose there exists a connection one-form $A$ such that the horizontal distribution $\mathrm{H}$ is integrable. Then $\pi \colon P\to M$ is a trivial bundle and there exists a global section $\chi \colon M\to P$ such that $\chi^\ast A = 0$.
\end{lemma}

Refer to Chapter II in \cite{K.N1996} for the proof.

\begin{proof}[Proof of Theorem~\ref{theorem:LA_exact_symp}]
    By Lemma~\ref{lemma:Arnold_tangent_fields}, $M_\lambda$ is diffeomorphic to $\TT^k \times \RR^{n-k}$. Without loss of generality, assume that $X_{f_1}, \ldots, X_{f_k}$ are infinitesimal generators of $\Sp^1$-actions whose flows have period 1, and that $X_{g_{k+1}}, \ldots, X_{g_n}$ are infinitesimal generators of $\RR$-actions (see Lemma~\ref{lemma:linear_combination_functions}). 
    
    Let $\Ldist = \ker \theta$ and $\overline{U} = \left\{x\in U \mid f_i (x)\neq 0 \ \forall\, i \text{ and } \theta (x) \neq 0\right\}$. By hypothesis, $F\colon U \to V$ is a trivial bundle. Hence, $U\cong V \times \TT^k \times \RR^{n-k}$ can be endowed with a Riemannian metric $g$ given by the product of a flat metric on $V\subseteq \RR^n$ and flat and invariant metrics on $\TT^k$ and $\RR^{n-k}$. The metric $g$ is thus flat and invariant by the Lie group action of $\TT^k \times \RR^{n-k}$. Consider the distribution
    \begin{equation}
        \Ldist^\theta = \big(\Ldist \cap \gen{X_{f_i}}_{i=1}^n\big)^{\perp_g} \cap \Ldist\, ,
    \end{equation}
    where $\perp_g$ denotes the orthogonal complement with respect to the metric $g$. This distribution is:
    \begin{enumerate}
        \item invariant by the Lie group action of $\TT^k \times \RR^{n-k}$,
        \item contained in $\Ldist$,
        \item complementary to the vertical bundle.
    \end{enumerate} 
    Indeed,
    \begin{equation}
        \ker \T F = \bigcap_{i=1}^n \ker \dd f_i = \gen{X_{f_i}}_{i=1}^n\, ,
    \end{equation}
    and 
    \begin{equation}
        \Ldist^\theta_x \oplus \gen{X_{f_i}(x)}_{i=1}^n = \T_x M\, ,
    \end{equation}
    for every $x\in \overline{U}$. Moreover, $F\colon \overline{U} \to \overline{U}/(\TT^k \times \RR^{n-k})$ is a principal bundle and $\Ldist^\theta$ is a principal connection with connection one-form $\theta$. In addition, the fact that $\theta\wedge \dd \theta = 0$ implies that $\Ldist$ is integrable. The fact that the orthogonal complement of an integrable distribution with respect to a flat metric is integrable implies that $\Ldist^\theta$ is integrable. Let $\hat{U}\subseteq \overline{U}$ be an open subset of $\overline{U}$ such that $\hat{V} = F(\hat{U})$ is simply connected. By Lemma~\ref{lemma:global_section}, there exists a global section $\chi$ of $F\colon \hat{U} \to \hat{V} \cong \hat{U}/(\TT^k \times \RR^{n-k})$ such that $\chi^\ast \theta = 0$.

    Let $\Phi\colon \TT^{k} \times \RR^{n-k} \times M \to M$ denote the Abelian Lie group action defined by the flows of $X_{f_i}$. For each point $x\in M_{\lambda} = F^{-1}(\lambda)$, the angle coordinates $(y^i(x))$ are determined by
    \begin{equation}
        \Phi\big(y^i(x), \chi(F(x))\big) = x\, .
    \end{equation} 
    Notice that $(y^i, f_i)$ are coordinates in $\hat{U}$ adapted to the foliation of $M$ in $M_\lambda$. 
    In these coordinates,
    \begin{equation}
        \theta = A_i \dd y^i + B^i \dd f_j\, , \quad X_{f_i} = \frac{\partial}{\partial y^i}\, ,
    \end{equation}
    for some functions $A_i$ and $B^i$. Contracting $\theta$ with $X_{f_i}$ yields $A_i = f_i$. 

    Finally, notice that the image of the section $\chi$ is given by the intersection of level sets of angle coordinates $y^i$, namely, $\Image \chi = \cap_{i=1}^n (y^i)^{-1}(\mu_i)$. Hence, 
    \begin{equation}
        0 = \chi^\ast \theta = B^i \dd f_i\, .
    \end{equation}
    Since $\mu_i$'s are arbitrary, we have that $B^i=0$. We conclude that $\theta = f_i \dd y^i$. 
\end{proof}

\section{Liouville--Arnold theorem for contact Hamiltonian systems}\label{sec:main_theorem}


Given the functions $f_\alpha\colon M\to \R,\ \alpha=0,\ldots, n$, define the map
\begin{equation}
    F =\left(f_0, \ldots, f_n\right)  \colon M \to \R^{n+1}\, .
\end{equation}
For each $\lambda\in \R^{n+1}\setminus\{0\}$, let $\ray{\lambda}$ denote the ray generated by $\lambda$, namely, 
$$ \ray{\lambda} = \{x \in \R^{n+1}\mid \exists\, r\in\R_+ \colon x=r \lambda\}\, .$$ 
Consider the subsets $M_\lambda = F ^{-1}(\lambda)$
and $ M_{\ray{\lambda}} = F^{-1}(\ray{\lambda})$ of $M$.


\begin{definition}\label{def:integrable_systems}
    A \emph{completely integrable contact system} is a triple $(M, \eta, F)$, where $(M, \eta)$ is a {(co-oriented)} contact manifold and $F=(f_0, \ldots, f_n)\colon M\to \RR^{n+1}$ is a map such that the functions $f_0, \ldots, f_n$ are in involution and $\T F$ has rank at least $n$ on $M$. The functions $f_0, \ldots, f_n$ are called \emph{integrals}.
\end{definition}

If there are singularities, such as critical points of the integrals, on a zero-measure subset $S$ of $M$, it suffices to restrict to the dense open subset $M_{0} = M\setminus S$.

See Remark~\ref{rem:characterization_integrable_systems} in Section~\ref{sec:coisotropic} for alternative characterizations of completely integrable contact systems.


\begin{theorem}[Liouville--Arnold theorem for contact Hamiltonian systems]\label{thm:main_theorem}
    Let $(M, \eta, F)$ be a completely integrable contact system, where $F=(f_0, \ldots, f_n)$.
    Assume that the Hamiltonian vector fields $X_{f_0}, \ldots, X_{f_n}$ are complete. 
    Given $\lambda \in \R^{n+1}\setminus\{0\}$,
    let $B \subseteq \RR^{n+1}\setminus\{0\}$ be an open neighbourhood of $\lambda$. 
    Given $\lambda\in \RR^{n+1}$, assume that $M_{\ray{\lambda}} = F^{-1}(\ray{\lambda})$ is connected (if not, replace $M_{\ray{\lambda}}$ by a connected component).
    Let $\pi\colon U\to M_{\ray{\lambda}}$ be a tubular neighbourhood of $M_{\ray{\lambda}}$ such that $\restr{F}{U}\colon U \to B$ is a trivial bundle over a domain $V \subseteq B$.
    Then:
    
    \begin{enumerate}
        \item The submanifold $M_{\ray{\lambda}}$ is coisotropic, invariant by the Hamiltonian flow of $f_\alpha$, and diffeomorphic to $\mathbb{T}^k\times \R^{n+1-k}$ for some $k\leq n$. In particular, if $M_\lambda$ is compact, then $k=n$. 
        \item There exist coordinates $(y^0, \ldots, y^n, \tilde A_1, \ldots, \tilde A_n)$ on $U$ such that the Hamiltonian vector fields of the functions $f_\alpha$ read
        \begin{equation}
            X_{f_\alpha} = \overline{N}_\alpha^\beta X_{f_\beta}\, ,
        \end{equation}
        where $\overline{N}_\alpha^\beta$ are functions depending only on $\tilde{A}_1, \ldots, \tilde{A}_n$.
        Hence, the integral curves of $X_{f_\beta}$ are given by
        \begin{equation}\label{eq:Hamilton_action_angle_symp}
            \begin{array}{ll}
                \dot y^\alpha = \Omega^\alpha(\tilde A_1, \ldots, \tilde A_n)\, ,\quad & \alpha=0, \ldots, n\, ,\\
                \dot{\tilde{A}}_i = 0\, ,  &i=1, \ldots, n\, ,
            \end{array}
        \end{equation}
        where $\Omega^\alpha\coloneqq \overline{N}^\alpha_\beta$.
        \item There exists a nowhere-vanishing function $A_0\in \Cinfty(U)$ and a conformally equivalent contact form $\tilde \eta = \eta/A_0$ such that $(y^i, \tilde A_i, y^0)$ are Darboux coordinates for $(M, \tilde \eta)$, namely, $\tilde{\eta} = \dd y^0 - \tilde{A}_i \dd y^i$.
    \end{enumerate}
\end{theorem}

The coordinates $(y^\alpha)$ and $(\tilde{A}_i)$, for $\alpha=0,\dotsc, n$ and $i=1,\dotsc, n$, are called \emph{angle coordinates} and \emph{action coordinates}, respectively. In physical applications, one of the integrals is regarded as the Hamiltonian function (which represents the energy of the system), and equations \eqref{eq:Hamilton_action_angle_symp} are the equations of motion.

There are functions $A_0, \ldots, A_n\in \Cinfty(U)$ such that the action coordinates are given by $\tilde{A_i}=-A_i/A_0\, (i=1, \ldots, n)$. These functions are the projection of the action coordinates from the symplectization of the completely integrable contact system (see Subsection~\ref{subsec:proof} for more details). In terms of these functions and the angle coordinates $(y^i)$, the original contact form reads $\eta = A_i \dd y^i$.


\begin{remark}
    In order to compute the action-angle coordinates for $(M, \eta, F)$, one has to carry out the following steps:
    \begin{enumerate}
        \item Construct the symplectization $(M^\Sigma, \theta, F^\Sigma)$, where $M^\Sigma = M\times \RR_+, \, \theta = r \Sigma^\ast \eta$ and $F^\Sigma = -r \Sigma^\ast F$, with $\Sigma \colon M^\Sigma \to M$ the canonical projection of $\RR_+$.
        \item Compute the action-angle coordinates $(y^\alpha_\Sigma, A_\alpha^\Sigma)\, , \alpha \in \{0, \ldots, n\}$ for $(M^\Sigma, \theta, F^\Sigma)$ (see Remark~\ref{remark:construction_coords_exact}).
        \item The angle coordinates for $(M, \eta, F)$ are functions $y^\alpha$ such that $y^\alpha_\Sigma = \Sigma^\ast y^\alpha$. Similarly, one has the functions $A_\alpha$ such that $A_\alpha^\Sigma = -r \Sigma^\ast A_\alpha$. 
        \item At least one of the functions $(A_\alpha)$ is non-vanishing. By relabeling the indices if necessary, assume that $A_0\neq 0$. The action coordinates are given by $\tilde{A}_i = - A_i/A_0, \, i\in \{1, \ldots, n\}$. 
    \end{enumerate}
\end{remark}


The proof of Theorem~\ref{thm:main_theorem} will consist of the following steps:
\begin{enumerate}
    \item Symplectize the completely integrable contact system $(M, \eta, F)$, obtaining the homogeneous integrable system $(M^\Sigma, \theta, F^\Sigma)$.
    \item Ensure that sufficient conditions hold for $F$ so that the Liouville--Arnold theorem can be applied to $(M^\Sigma, \theta, F^\Sigma)$.
    \item Apply the Liouville--Arnold theorem for exact symplectic manifolds (Theorem~\ref{theorem:LA_exact_symp}).
    \item Construct action-angle coordinates for $(M, \eta, F)$ from the ones for $(M^\Sigma, \theta, F^\Sigma)$.
\end{enumerate}

It is worth remarking that, in order to obtain coordinates on $M$ from the ones on $M^\Sigma$, we need the coordinates on $M^\Sigma$ to be homogeneous functions of degree 1. This is not guaranteed if one applies the Liouville--Arnold theorem for non-compact invariant submanifolds by Fiorani, Giachetta and Sardanashvily \cite{Fiorani2003,Fiorani2003a}. We need to apply Theorem~\ref{theorem:LA_exact_symp} instead.

\subsection{Symplectization of contact manifolds}\label{sec:symplectization}

In this subsection, we present the results concerning the symplectization of contact manifolds that will be employed subsequently. For more details, see \cite{Lainz2022,Grabowska2022,G.G2024a,Libermann1987}.

\begin{definition}
    Let $\cdist$ be a contact distribution on a manifold $M$. Consider an exact symplectic manifold $(M^\Sigma, \theta)$ and a (locally trivial) fiber bundle $\Sigma\colon M^\Sigma \to M$. This fiber bundle is called a \emph{symplectization} if
    \begin{equation}\label{eq:def_symplectization}
        {\T \Sigma} \left(\Ldist\right) = \cdist\, ,
    \end{equation}
    outside the singular points, where $\mathfrak L = \ker \theta$ is called the \emph{Liouville distribution}.
\end{definition}


\begin{proposition}
    Let $(M, \eta)$ and $(M^\Sigma, \theta)$ be a contact manifold and an exact symplectic manifold, respectively.
    A (locally trivial) fiber bundle $\Sigma\colon M^\Sigma \to M$ is a symplectization if and only if there exists a nowhere-vanishing function $\sigma\colon M^\Sigma \to \R$ such that
    \begin{equation}
        \sigma \cdot \left(\Sigma^\ast \eta \right)= \theta\, .
    \end{equation}
    The function $\sigma$ is called the \emph{conformal factor} of $\Sigma$.
\end{proposition}

\begin{proof}
    Since $\cdist$ and ${\T \Sigma} \left(\Ldist\right)$ are both distributions of corank $1$ on $M$, in order to ensure that $\Sigma$ is a symplectization it is enough to verify that ${\T \Sigma} \left(\Ldist\right) \subseteq \cdist$.
    By taking duals, this condition is equivalent to 
    \begin{equation}
        \Sigma^\ast (\cdist^\circ) \subseteq \Ldist^\circ\, ,
    \end{equation}
    but $\cdist^\circ= \langle\eta\rangle$ and $\Ldist^\circ = \langle\theta\rangle$. Thus,
    \begin{equation}\label{eq_symp}
        \sigma \cdot \left(\Sigma^\ast \eta\right) = \theta\, , 
    \end{equation}
    for some nowhere-vanishing function $\sigma$. 
    

\end{proof}




\begin{proposition}
    Let $(M, \eta)$ be a {(co-oriented)} contact manifold, $(M^\Sigma, \theta)$ an exact symplectic manifold, and $\Sigma \colon M^\Sigma \to M$ a symplectization with conformal factor $\sigma$. Then, $\ker \T \Sigma = \langle \Delta\rangle$, where $\Delta$ denotes the Liouville vector field of $(M^\Sigma, \theta)$. Moreover, the conformal factor is a homogeneous function of degree $1$, namely, $\Delta(\sigma) = \sigma$. 
\end{proposition}

\begin{proof}
    Taking into account that the rank of $\Im \T \Sigma$ is $2n+1$, the rank of $\ker \T \Sigma$ must be $1$. Therefore, the proof is completed by showing that $\Delta$ belongs to the distribution $\ker \T \Sigma$.

    Observe that $\contr{\Delta} \theta = 0$, and thus $\contr{\Delta} (\Sigma^\ast \eta) = 0$. On the other hand, since $\theta = \contr{\Delta} \dd \theta$, we have that
    \begin{equation}\label{eq:eta_contr_Delta}
        \sigma \cdot \left( \Sigma^\ast \eta \right) = \contr{\Delta} \dd \big(\sigma \cdot \left(\Sigma^\ast \eta\right)\big) 
        = \Delta(\sigma) \left(\Sigma^\ast \eta\right) + \sigma \contr{\Delta} \left(\Sigma^\ast \dd \eta\right)
    \end{equation}
    Let $\Reeb$ be the Reeb vector field of $(M, \eta)$. Consider a vector field $\tilde{\Reeb}\in \X(M^\Sigma)$ such that ${\T \Sigma} \left(\tilde{\Reeb}\right) = \Reeb$. Then, 
    \begin{equation}
        \contr{\tilde{\Reeb}} (\Sigma^\ast \eta) = \Sigma^\ast (\contr{\Reeb} \eta) = 1  \, , \quad
       \contr{\tilde{\Reeb}} (\Sigma^\ast \dd \eta) = \Sigma^\ast (\contr{\Reeb} \dd \eta) = 0 \, . 
    \end{equation}
    Contracting both sides of equation \eqref{eq:eta_contr_Delta} with $\tilde{\Reeb}$ yields $\Delta(\sigma) = \sigma$, which in turn implies that 
    \begin{equation}
        0 = \contr{\Delta} (\Sigma^\ast \dd \eta) = \Sigma^\ast \left(\contr{{\T \Sigma}\left( \Delta\right)} \dd \eta\right)\, .
    \end{equation}
    Since $\Sigma$ is a submersion, $\contr{{\T \Sigma} \left(\Delta\right)} \dd \eta = 0$. Therefore, $\flat\Big({\T \Sigma} \left(\Delta\right)\Big) = 0$, and hence ${\T \Sigma} \left(\Delta\right) = 0$.
\end{proof}

\begin{theorem}[Symplectization of {conformal contactomorphisms}]\label{thm:symplectization_contactomorphism}

    Let $(M_1, \eta_1)$ and $(M_2, \eta_2)$ be {(co-oriented)} contact manifolds. 
    Consider {two} symplectizations $\Sigma_1 \colon M_1^\Sigma \to M_1$ and $\Sigma_2\colon M_2^\Sigma \to M_2$ with symplectic potentials $\theta_1 \in \Omega^1(M_1^\Sigma)$ and $\theta_2 \in \Omega^1(M_2^\Sigma)$, and conformal factors $\sigma_1\in \Cinfty(M_1^\Sigma)$ and $\sigma_2\in \Cinfty(M_2^\Sigma)$, namely, $\sigma_1 \Sigma_1^\ast \eta_1 = \theta_1$ and $\sigma_2 \Sigma_2^\ast \eta_2 = \theta_2$.
    
    If $F^\Sigma \colon M_1^\Sigma \to M_2^\Sigma$ is a homogeneous symplectomorphism, then it can be projected onto a map $F\colon M_1 \to M_2$ making the diagram
    \begin{equation}\label{eq:symplectization_morphisms_diagram}
        \begin{tikzcd}
            M_1^\Sigma \arrow[r, "F^\Sigma"] \arrow[d, "\Sigma_1"] & M_2^\Sigma \arrow[d, "\Sigma_2"] \\
            M_1 \arrow[r, "F"]                                   & M_2                           
            \end{tikzcd}
    \end{equation}
    commutative. The map $F\colon M_1 \to M_2$ is a {conformal contactomorphism}. 
    
    Conversely, if $F\colon M_1 \to M_2$ is a {conformal contactomorphism} with conformal factor $f\in \Cinfty (M_1)$ (i.e., $F^\ast \eta_2 = f \eta_1$), {then there is a unique homogeneous symplectomorphism $F^\Sigma \colon M_1^\Sigma \to M_2^\Sigma$ such that the diagram~\eqref{eq:symplectization_morphisms_diagram} commutes. This map can be characterized among those making the diagram~\eqref{eq:symplectization_morphisms_diagram} commutative by the following identity: 
    \begin{equation}
        {(F^\Sigma)}^\ast(\sigma_2) = \frac{\sigma_1}{\Sigma_1^\ast(f)}\, .
    \end{equation}
    The map $F^\Sigma$ is called the \emph{symplectization} of $F$.}
    

    Moreover, the symplectization is functorial. More specifically, if $(M_3, \eta_3)$ is a {(co-oriented)} contact manifold, $(M_3^\Sigma, \theta_3)$ an exact symplectic manifold, and $G\colon M_2 \to M_3$ a {conformal contactomorphism}, then the symplectization $G^\Sigma\colon M_2^\Sigma \to M_3^\Sigma$ of $G$ verifies ${(G \circ F)}^\Sigma = G^\Sigma \circ F^\Sigma$. Hence, the following diagram is commutative:
    \begin{equation}\label{eq:diagram_functioral_symplectization}
        \begin{tikzcd}
            M^\Sigma_1 \arrow[r, "F^\Sigma"'] \arrow[d, "\Sigma_1"] \arrow[rr, "(G \circ F)^\Sigma", bend left] & M^\Sigma_2 \arrow[d, "\Sigma_2"] \arrow[r, "G^\Sigma"'] & M^\Sigma_3 \arrow[d, "\Sigma_3"] \\
            M_1 \arrow[r, "F"] \arrow[rr, "G \circ F"', bend right]                                           & M_2 \arrow[r, "G"]                                    & M_3                 
            \end{tikzcd}
    \end{equation}
\end{theorem}

\begin{proof}

Let $\Delta_1$ and $\Delta_2$ denote the Liouville vector fields of $(M_1^\Sigma, \theta_1)$ and $(M_2^\Sigma, \theta_2)$, respectively. Let $\cdist_1 = \ker \eta_1$ and $\cdist_2 = \ker \eta_2$ denote the contact distributions, and $\Ldist_1 = \ker \theta_1$ and $\Ldist_2 = \ker \theta_2$ the Liouville distributions.

Suppose that $F^\Sigma\colon M_1^\Sigma\to M_2^\Sigma$ is a homogeneous symplectomorphism. Then, ${(F^\Sigma)}_\ast\Delta_1 = \Delta_2$, $\ker \T \Sigma_1 = \gen{\Delta_1}$ and $\ker \T \Sigma_2 = \gen{\Delta_2}$. This implies that the map $F^\Sigma$ projects onto a map $F$ such that the diagram \eqref{eq:symplectization_morphisms_diagram} commutes. Using the commutative diagram, one obtains
\begin{equation}
    F_\ast \cdist_1 = {\T \Sigma_2 \circ (F^\Sigma)_\ast \circ (\T \Sigma_1)^{-1}} (\cdist_1) = \T \Sigma_2 \circ (F^\Sigma)_\ast\, \Ldist_1 = 
    \T \Sigma_2(\Ldist_2) = \cdist_2\, .
\end{equation}
Thus, $F$ is a {conformal contactomorphism}.

Conversely, suppose that $F\colon M_1 \to M_2$ is a {conformal contactomorphism} with conformal factor $f$. {Let $F^\Sigma\colon M_1^\Sigma \to M_2^\Sigma$ be a map making the diagram~\eqref{eq:symplectization_morphisms_diagram} commutative}.
Then,
\begin{equation}
    \begin{aligned}
        {(F^\Sigma)}^\ast \theta_2 &= {(F^\Sigma)}^\ast \big(\sigma_2 \left(\Sigma_2^\ast \eta_2\right) \big)
        ={(F^\Sigma)}^\ast \left[\sigma_2 \Big(\big(F\circ \Sigma_1 \circ (F^\Sigma)^{-1}\big)^\ast \eta_2\Big)\right] 
        \\ & 
        = {(F^\Sigma)}^\ast(\sigma_2) \left(\Sigma_1^\ast F^\ast \eta_2\right) 
        =
        {(F^\Sigma)}^\ast(\sigma_2) \left(\Sigma_1^\ast (f \eta_1)\right) \\ &=
        {(F^\Sigma)}^\ast(\sigma_2) \Sigma_1^\ast(f) \frac{1}{\sigma_1} \theta_1\, ,
    \end{aligned}
\end{equation}
and hence ${(F^\Sigma)}^\ast \theta_2 = \theta_1$ if and only if
\begin{equation}
    \left(F^\Sigma\right)^\ast(\sigma_2) = \frac{\sigma_1}{\Sigma_1^\ast(f)}\, .
\end{equation}
{Furthermore, any system of coordinates $(x_k^i)_i$ on $M_k$ (with $k=1, 2$) induces a system of coordinates $(\Sigma^*x_k^i, \sigma_k)$  are a set of coordinates for $M^\Sigma_k$. Hence, $F^\Sigma$ is fully determined by $F$ and the value of ${(F^\Sigma)}^\ast(\sigma_2)$.}
This implies that each {conformal contactomorphism} has a unique symplectization. 



The commutativity of the diagram \eqref{eq:diagram_functioral_symplectization} follows from the fact that the composition of {conformal contactomorphisms} (respectively, exact symplectomorphisms) is a {conformal contactomorphism} (respectively, exact symplectomorphism), and the uniqueness of the symplectization of each {conformal contactomorphism}.
\end{proof}



The fact that the symplectization preserves the composition of morphisms permits to symplectize the vector fields. 
Let $X\in \X(M)$ be an infinitesimal {conformal contactomorphism}, and let $\phi_t$ be its flow. Notice that
\begin{equation}
 {\phi}_{t+r}^\Sigma = {(\phi_r \circ \phi_t)}^\Sigma =  {\phi}^\Sigma_t \circ {\phi}^\Sigma_r\, , 
\end{equation}
where these compositions are defined,
and thus ${\phi}_t^\Sigma$ is also a flow. Since it is the symplectization of {conformal contactomorphisms}, $\phi_t^\Sigma$ is a flow of homogeneous symplectomorphisms. Hence, its infinitesimal generator, ${X}^\Sigma$ is an infinitesimal homogeneous symplectomorphism. 
\begin{theorem}[Symplectization of vector fields]\label{thm:symp_fields}
    Let $\cdist$ be a contact distribution on a manifold $M$. Let $(M^\Sigma, \theta)$ be an exact symplectic manifold. A symplectization $\Sigma\colon M^\Sigma\to M$ provides a bijection between infinitesimal {conformal contactomorphisms} $X$ on $M$ and infinitesimal homogeneous symplectomorphisms $X^\Sigma$ on $M^\Sigma$ such that $X$ and $X^\Sigma$ are related by the bijection if and only if
    \begin{equation}
        {\T \Sigma} (X^\Sigma) = X\, .
    \end{equation}
    If $(M, \eta)$ is a {(co-oriented)} contact manifold with $\ker \eta = \cdist$, the vector field $X^\Sigma$ can also be characterized as the unique one that projects onto $X$ and satisfies
    \begin{equation}
        X^\Sigma(\sigma) = -\left(a_X\circ \Sigma\right)\cdot \sigma\, ,
    \end{equation}
    where $a_X$ is the conformal factor of $X$, i.e., $\liedv{X}\eta = a_X\eta$.
    We can also write it as
    \begin{equation}
        X^\Sigma = X^\sigma - \left(a_X\circ \Sigma\right)  \Delta\, ,
    \end{equation}
    where $X^\sigma$ is the unique vector field that projects onto $X$ and satisfies $X^\sigma(\sigma) = 0$.
    In particular, for the Reeb vector field $\Reeb \in \X(M)$ associated with $\eta$,
    \begin{equation}
        {\Reeb}^\Sigma = \Reeb^\sigma\, .
    \end{equation}
    Moreover, for any pair of infinitesimal {conformal contactomorphisms} $X$ and $Y$ on $M$, we have
    \begin{equation}\label{eq:symplectization_fields}
        \lieBr{X,Y}^\Sigma = \lieBr{X^\Sigma, Y^\Sigma}\, . 
    \end{equation} 
\end{theorem}

\begin{proof}
   By Theorem \ref{thm:symplectization_contactomorphism}, given the flow $\phi$ of an infinitesimal {conformal contactomorphism} $X$, there exists a unique flow $\phi^\Sigma$ which is made of exact symplectomorphisms and projects onto $\phi$. Hence, there exists a unique infinitesimal homogeneous symplectomorphism $X^\Sigma$, which is the infinitesimal generator of $\phi^\Sigma$. 
    Similarly, one can see that every infinitesimal homogeneous symplectomorphism projects onto an infinitesimal {conformal contactomorphism}.

    Consider the contact form $\eta$ generating $\cdist$ with conformal factor $\sigma$ (i.e., $\sigma \cdot \left(\Sigma^\ast \eta\right) = \theta$) and let $X$ be an infinitesimal {conformal contactomorphism} with conformal factor $a_X$ (i.e., $\liedv{X} \eta= a_X \eta$).
    {The associated infinitesimal homogeneous symplectomorphism $X^\Sigma$ satisfies
    \begin{align*}
         0 & = \liedv{X^\Sigma} \theta = \liedv{X^\Sigma}\big(\sigma \cdot \left(\Sigma^\ast\eta\right)\big) =
        X^\Sigma(\sigma) \left(\Sigma^\ast\eta\right)  + \sigma\cdot \liedv{X^\Sigma} \left(\Sigma^\ast\eta\right)\\
        & = X^\Sigma(\sigma) \left(\Sigma^\ast\eta\right)  + \sigma\cdot \Sigma^\ast \left( \liedv{X} \eta\right)
        = \left({X^\Sigma}(\sigma)  + \sigma\cdot \left(a_X\circ \Sigma\right)\right) \left(\Sigma^\ast\eta\right)\, ,
    \end{align*}
   which implies that ${X^\Sigma}(\sigma)  = - \sigma\cdot \left(a_X\circ \Sigma\right)$.}

    Since $\ker \T \Sigma = \gen{\Delta}$, the vector field $X^\Sigma$ is of the form $X^\sigma + b\Delta$ for some function $b\in \Cinfty(M^\Sigma)$. Thus,
    \begin{equation}
        -\left(a_X\circ \Sigma\right)\cdot \sigma  = X^\Sigma(\sigma) = b \Delta (\sigma) = b \sigma\, ,
    \end{equation}
    since $\sigma$ is homogeneous of degree $1$. Hence, $b = -\left(a_X\circ \Sigma\right)$.
     
    Finally, {let $\psi_t$ be the flow of $Y$.}
    {The Lie bracket can be written as (see \cite[p.~286]{Abraham1988})
    \begin{align*}
        [X^\Sigma,\,Y^\Sigma](x)
        & =\restr{\frac{\partial}{\partial \varepsilon}}{\varepsilon=0}\Big(
        \left(\psi^\Sigma_{-\sqrt{\varepsilon}}\circ\phi^\Sigma_{-\sqrt{\varepsilon}}\circ\psi^\Sigma_{\sqrt{\varepsilon}}\circ\phi^\Sigma_{\sqrt{\varepsilon}}\right)(x)\Big)\\
        & =\restr{\frac{\partial}{\partial \varepsilon}}{\varepsilon=0}\Big(
        \left(\psi_{-\sqrt{\varepsilon}}\circ\phi_{-\sqrt{\varepsilon}}\circ\psi_{\sqrt{\varepsilon}}\circ\phi_{\sqrt{\varepsilon}}\right)^\Sigma(x) \Big)
        = [X,Y]^\Sigma(x)\, , \quad \forall\, x \in M^\Sigma\, .
    \end{align*}}
\end{proof}

Since the contact Hamiltonian vector fields and the functions on $(M, \eta)$ on one side, and the homogeneous Hamiltonian vector fields and the homogeneous functions on $(M^\Sigma, \theta)$ on the other side, are in one-to-one correspondence, there is a bijection between both sets of functions.

\begin{theorem}\label{thm:symp_functions}
    Let $(M, \eta)$ be a {(co-oriented)} contact manifold and $(M^\Sigma, \theta)$ an exact symplectic manifold.
    For each $f\in \Cinfty(M)$, let $X_f\in \X(M)$ denote the Hamiltonian vector field with respect to $\eta$. 
    For each $f^\Sigma\in \Cinfty(M^\Sigma)$, let ${X_{f^\Sigma}\in \X(M^\Sigma)}$ denote the Hamiltonian vector field with respect to $\theta$. 
    Given a symplectization $\Sigma\colon M^\Sigma \to M$ with conformal factor $\sigma$, there is a bijection between functions $f$ on $M$ and homogeneous functions of degree 1 $f^\Sigma$ on $M^\Sigma$ such that
    \begin{equation}
        {\T \Sigma} \left(X_{f^\Sigma}\right) = X_f\, .
    \end{equation}
    This bijection is given by
    \begin{equation}
        f^\Sigma = -\sigma \cdot \left(\Sigma^\ast f\right)\, .
    \end{equation}
    Moreover, one has
    \begin{equation}
        \left\{f^\Sigma, g^\Sigma \right\}_\theta = \left\{f, g \right\}^\Sigma_\eta\, ,
    \end{equation}
    where $\{\cdot, \cdot\}_\theta$ and $\{\cdot, \cdot\}_\eta$ denote the Poisson bracket of $(M^\Sigma, \theta)$ and the Jacobi bracket of $(M, \eta)$, respectively. Note that this implies that $\Sigma$ is a conformal Jacobi morphism, i.e., $\left\{\sigma \cdot \left(\Sigma^\ast f\right), \sigma \cdot \left(\Sigma^\ast g\right)\right\}_\theta = \sigma \cdot \left( \Sigma^\ast \{f, g\}_\eta\right)$.
    
\end{theorem}

\begin{proof}
    Indeed,
    \begin{equation}
        f^\Sigma = \theta(X_{f^\Sigma}) = \theta({(X_f)^\Sigma })
        = \sigma \cdot \left( \Sigma^\ast(\eta)((X_f)^\Sigma )\right) = \sigma \cdot \left(\Sigma^\ast(\eta(X_{f}))\right) = -\sigma \cdot \left(\Sigma^\ast f\right)\, .
    \end{equation}
    The other claims follow from Proposition \ref{prop:hamiltonian_bijection}, together with the equivalent result for symplectic geometry (i.e., $X_{\jacBr{f^\Sigma,g^\Sigma}_\theta} = - \lieBr{X_{f^\Sigma},X_{g^\Sigma}}$), and the fact that symplectization preserves the Lie brackets (equation~\eqref{eq:symplectization_fields}).
\end{proof}

\begin{corollary}\label{corollary:dissipated_conseved}
    A function $f$ on $M$ is a dissipated quantity with respect to $(M, \eta, h)$ if and only if $f^\Sigma$ is a conserved quantity with respect to $(M^\Sigma, \theta, h^\Sigma)$.
\end{corollary}


\begin{remark}
    Note that a function $\hat{f}: M^\Sigma \to \R$ is homogeneous of degree $0$ if and only if one can write $\hat{f} = \Sigma^\ast f$ for some $f:M \to \R$. This follows from Theorem~\ref{thm:symp_functions} and the fact that a function $\hat{f}$ is homogeneous of degree $0$ if and only if the function $-\sigma \hat{f}$ is homogeneous of degree $1$.
\end{remark}

Additionally, there is a one-to-one correspondence between homogeneous integrable systems and completely integrable contact systems.

\begin{proposition}
    Let $(M, \eta)$ be a contact manifold. Suppose that $\Sigma\colon M^\Sigma \to M$ is a symplectization such that $\theta = \sigma \cdot (\Sigma^\ast \eta)$ is the symplectic potential on $M^\Sigma$. Then, $(M^\Sigma, \theta, F^\Sigma)$, with $F^\Sigma=-\sigma \cdot (\Sigma^\ast F)$, is a homogeneous integrable system if and only if $(M, \eta, F)$ is a completely integrable contact system.
\end{proposition}

\subsection{Proof of the Liouville--Arnold theorem for contact Hamiltonian systems}\label{subsec:proof}

In this subsection, Theorem~\ref{thm:main_theorem} is proven.
{We will start by symplectizing the completely integrable contact system $(M, \eta, F)$, obtaining the homogeneous integrable system $(M^\Sigma, \theta, F^\Sigma)$, and ensuring that $(M^\Sigma, \theta, F^\Sigma)$ satisfies the hypothesis of Theorem~\ref{theorem:LA_exact_symp}). After that, we will construct action-angle coordinates for $(M, \eta, F)$ from the ones for $(M^\Sigma, \theta, F^\Sigma)$.}


Hereinafter, consider the symplectization $\Sigma\colon M^\Sigma =  M\times \R_+\to M$, where $\Sigma$ is the canonical projection on the first component of the product manifold, and $\R_+$ denotes the positive real half-line. Its conformal factor is $\sigma=s$, with $s$ the global coordinate of $\R_+$.
    Let $F^\Sigma=(f_0^\Sigma, \ldots, f_n^\Sigma)\colon M\times\R_+\to \R^{n+1}$, where  $f_\alpha^\Sigma = -\sigma \cdot \left(\Sigma^\ast f_\alpha\right)$. Let $\lambda\in \R^{n+1}\setminus\{0\}$. Then,
    \begin{equation}
        (F^\Sigma)^{-1}(\lambda) = \{(x,s)\in  M\times\R_+\mid - s F \circ \Sigma(x,s) = \lambda\} = \left\{(x,s)\in  M\times\R_+\mid  F (x) = -\frac{\lambda}{s}\right\}\, ,
    \end{equation}
    which implies that
    \begin{equation}
        \Sigma \big( (F^\Sigma)^{-1}(\lambda) \big) 
        = \left\{x \in  M \mid \exists\, s\in \R_+\colon\   F (x) = -\frac{\lambda}{s}\right\} 
        = M_{\ray{-\lambda}}\, .
    \end{equation}
    Since $X_{f_\alpha^\Sigma}$ are tangent to $(F^\Sigma)^{-1}(\lambda)$, the Hamiltonian vector fields $X_{f_\alpha}={\T \Sigma} \left(X_{f_\alpha^\Sigma}\right)$ are tangent to $M_{\ray{-\lambda}}$.
    
    {Let us recall that we are assuming that $\lambda\neq 0$.}
    Since $\dd f_\alpha^\Sigma = \sigma \cdot \left(\Sigma^\ast \dd f_\alpha\right) + \Sigma^\ast \left(f_\alpha\right) \dd \sigma$ and $\sigma$ is functionally independent of $\Sigma^\ast f_\alpha$, it suffices that $\dd f_\alpha$ have rank $n$ for $\dd f_\alpha^\Sigma$ to have rank $n+1$. The fact that $f_\alpha^\Sigma$ are functionally independent implies that $X_{f_\alpha^\Sigma}$, and then $X_{f_\alpha}$, are linearly independent. On the other hand, since $f_\alpha$ are in involution, their Hamiltonian vector fields commute. Hence, by Theorem~\ref{lemma:Arnold_tangent_fields}, $M_{\ray{\lambda}}$ is diffeomorphic to $\TT^k\times \R^{n+1-k}$ for some $k\leq n+1$.

    The tangent space $\T M_{\ray{\lambda}}$ is spanned by $X_{f_{\alpha}}$. Hence, its annihilator, $\T M_{\ray{\lambda}}^\circ$, is spanned by the one-forms $\Omega_{\alpha\beta}=f_\alpha \dd f_\beta - f_\beta \dd f_\alpha$. Indeed, using the identity~\eqref{eq:Hamiltonian_vf_Jacobi_bracket}, we can write
    \begin{equation}
        \contr{X_{f_\gamma}} \Omega_{\alpha \beta} = f_\alpha X_{f_\gamma} (f_\beta) - f_\beta X_{f_\gamma} f_\alpha
        = f_\alpha \{f_\gamma, f_\beta\} - f_\beta \{f_\gamma, f_\alpha\} = 0\, ,
    \end{equation}
    for any $\alpha, \beta, \gamma \in \{0, \ldots, n\}$. Since $\T F$ has rank at least $n$, there are at least $n$ independent forms $\Omega_{\alpha\beta}$. Thus, $\T M_{\ray{\lambda}}^{\perp_\lambda}$ is spanned by the vector fields $Y_{\alpha \beta}= \lsharp(\Omega_{\alpha \beta})$. Then, 
    \begin{align}
        \Omega_{\alpha \beta} (Y_{\gamma \delta}) 
        & = f_\alpha f_\gamma \lambda ( \dd f_\beta, \dd f_\delta )
        - f_\beta f_\gamma \lambda ( \dd f_\alpha, \dd f_\delta )
        - f_\alpha f_\delta \lambda ( \dd f_\beta, \dd f_\gamma )
        + f_\beta f_\delta \lambda ( \dd f_\alpha, \dd f_\gamma ) \\
        & = f_\alpha f_\gamma \left\{ f_\beta,  f_\delta \right\}
        - f_\beta f_\gamma \left\{  f_\alpha,  f_\delta  \right\}
        - f_\alpha f_\delta \left\{ f_\beta,  f_\gamma  \right\}
        + f_\beta f_\delta \left\{  f_\alpha,  f_\gamma  \right\} = 0,
    \end{align}
    which implies that $\T M_{\ray{\lambda}}^{\perp_\lambda} \subseteq \T M_{\ray{\lambda}}$, i.e., $M_{\ray{\lambda}}$ is coisotropic. 

    Since $f_\alpha$'s are in involution with respect to the Jacobi bracket, $f_\alpha^\Sigma$ are in involution with respect to the Poisson bracket. The Hamiltonian vector fields $X_{f_\alpha^\Sigma}$ are complete since $X_{f_\alpha}$ are complete. 
    Observe that the map $F^\Sigma$ is given by $F^\Sigma(x, r) = -r F(x)$, for each $(x, r) \in M \times \RR_+$.
    Let $\tilde{U} = \Sigma^{-1}(U)$. Since the map $\restr{F}{U}\colon U \to B$ is a trivial bundle over a domain $V \subseteq B$, the map $\restr{F^\Sigma}{\tilde{U}}\colon \tilde{U} \to B$ is also a trivial bundle over $V$.
    Finally, as it has been proven above, $\dd f_\alpha^\Sigma$ has rank $n+1$.

    {By the preceeding paragraph, all the hypotheses of Theorem~\ref{theorem:LA_exact_symp} hold for the functions $f_\alpha^\Sigma$ on the exact symplectic manifold $(M^\Sigma, \theta)$.} Hence, there are coordinates $(y^\alpha_\Sigma, A_\alpha^\Sigma)$ on $\tilde U$ such that 
    \begin{equation}
        \theta = A_\alpha^\Sigma \dd y^\alpha_\Sigma\, ,\qquad
        A_\alpha^\Sigma = M^\beta_\alpha f_\beta^\Sigma\, , 
    \end{equation}
    where $M_\alpha^\beta$ are homogeneous functions of degree $0$ depending only on the functions $(f_0^\Sigma, \ldots, f_n^\Sigma)$ (see Lemma~\ref{lemma:linear_combination_functions}), and 
    \begin{equation}
        X_{f_\alpha^\Sigma} = N_{\alpha}^\beta \frac{\partial}{\partial y^\beta_\Sigma}\, , 
    \end{equation}
    where $(N^\alpha_\beta)$ denotes the inverse matrix of $(M^\alpha_\beta)$.

    {Finally, we will construct action-angle coordinate in $M$ from the homogeneous action-angle coordinates in $M^\Sigma$.}
    By construction, $A_\alpha^\Sigma$ are homogeneous functions of degree 1, and therefore there exist functions $A_\alpha$ on $M$ such that $A_\alpha^\Sigma = -\sigma \cdot \left(\Sigma^\ast A_\alpha\right)$. 
    Similarly, since $y^\alpha_\Sigma$ are homogeneous of degree 0, there exist functions $y^\alpha$ on $M$ such that $y^\alpha_\Sigma = \Sigma^\ast y^\alpha$. 
    Since $\sigma \cdot \left(\Sigma^\ast \eta\right) = \theta$, the contact form is given by
    \begin{equation}
        \eta = A_\alpha \dd y^\alpha\, .
    \end{equation}
    The functions $A_\alpha$ are dissipated quantities since they are linear combinations of dissipated quantities. Moreover,
    \begin{equation}
        f_\alpha = \overline{M}^\beta_\alpha A_\beta\, , \quad X_{f_\alpha} = \overline{N}_\alpha^\beta \frac{\partial}{\partial y^\beta}\, ,
    \end{equation}
    where $\overline{M}_\alpha^\beta$ and $\overline{N}_\alpha^\beta$ are functions on $M$ such that $M_\alpha^\beta = \Sigma^\ast \overline{M}_\alpha^\beta$ and $N_\alpha^\beta = \Sigma^\ast \overline{N}_\alpha^\beta$. Hence, in coordinates $(y^\alpha, A_i)$ the contact Hamilton equations read
    \begin{equation}
    \begin{aligned}
        &\dot y^\alpha = \Omega^\alpha \, ,\\
        &\dot{\tilde{A}}_i = 0\, ,
    \end{aligned}
    \end{equation}
    where $\Omega^\alpha = \overline{N}^\alpha_\beta$ if $f_\beta\equiv h$ is the Hamiltonian function.

    Since $\lambda\neq 0$, there is at least one nonvanishing $f_\alpha$. Hence, there is at least one nonvanishing $A_\alpha$. By relabeling the $A_\alpha$ if necessary, without loss of generality one can assume that $A_0\neq 0$. Defining $\tilde A_i = -A_i/A_0$ and 
    $$\tilde \eta = \frac{1}{A_0} \eta = \dd y^0 - \tilde A_i \dd y^i,$$
    one can observe that $(y^i, \tilde A_i, y^0)$ are Darboux coordinates for the contact manifold $(M, \tilde \eta)$.
    \begin{flushright}
        \qedsymbol
    \end{flushright}

\section{Integrable systems, Legendrian and coisotropic submanifolds}\label{sec:coisotropic}
The aim of this section is to use the properties of coisotropic submanifolds in order to produce alternative characterizations of completely integrable contact systems. This provides a richer information about the geometry of contact completely integrable systems. In addition, these results will be useful to compare our results with existing literature in the next section.

First, we review how the Jacobi structure changes when we modify our contact form $\eta$ to another one $\bar{\eta}$ sharing the same contact distribution (see \cite{Dazord1991}). 

\begin{proposition}\label{prop:conformal_equivalence}
    Let $(M, \eta)$ be a contact manifold. Let $\bar{\eta} = a \eta$, where $a\colon M \to \R$ is nowhere-vanishing function, be a conformally equivalent contact form. Denote by $\JacLambda, \bar{\JacLambda}$ the bivectors of the Jacobi structures, by $X_f, \bar{X}_f$ the Hamiltonian vector fields of $f:M \to \R$, by $\Reeb, \bar{\Reeb}$ the Reeb vector fields, and by $\jacBr{\cdot, \cdot}_{\eta}, \jacBr{\cdot, \cdot}_{\bar{\eta}}$ the Jacobi brackets with respect to $\eta$ and $\bar{\eta}$, respectively. Then, 
    \begin{enumerate}
        \item $X_{f} = \bar{X}_{af}$,
        \item $-X_{1/a} = \bar{\Reeb}$,
        \item $\JacLambda = a\bar{\JacLambda}$,
        \item $\jacBr{f,g}_{\bar{\eta}} = a \jacBr{f/a, g/a}_{\eta}$.
    \end{enumerate}
    As a consequence of the third statement, the Jacobi complement coincides for conformally equivalent forms. Thus, a submanifold is coisotropic (resp.~Legendrian) for one contact structure if and only if it is coisotropic (resp.~Legendrian) for every other contact structure conformal to it.
\end{proposition}

\begin{proposition}\label{prop:coisotropic}
    Let $(M, \JacLambda, E)$ be a Jacobi manifold. A submanifold $N \hookrightarrow M$ is coisotropic if and only if for each point $x\in N$ there is a neighbourhood $U$ of $x$ in $M$ such that all functions $f,g:M \to \RR$ that are constant on $N\cap U$ satisfy
    \begin{equation}
        \restr{\JacLambda(\dd f, \dd g)}{N\cap U} = 0\, .
    \end{equation}
\end{proposition}

\begin{proof}
    Let $x\in N$ and let $U$ be a neighbourhood of $x$ in $M$ such that $N\cap U$ is a level set of functions $f_i\colon U\to \R,\ i=1,\ldots, k$. Then, $\T_x N = \ker \{\dd_x f_i\}_i$, in other words, $(\T_x N)^\circ = \langle \dd_x f_i \rangle$, and thus $(\T_x N)^{\perp_\JacLambda} = \langle \lsharp \dd_x f_i \rangle$. Hence, $(\T_x N)^{\perp_\JacLambda}\subseteq \T_x N$ if and only if
    \begin{equation}
        0 = \dd_x f_j (\lsharp \dd_x f_i) = \JacLambda(\dd_x f_i, \dd_x f_j)\, ,
    \end{equation}
    for every $i, j\in \{1, \ldots, k\}$.
    Since $x$ is arbitrary, the result follows.
\end{proof}

Let $S\colon \R^{n+1}\setminus \{0\} \to \Sp^n$ be the canonical projection on the sphere, that is, $S\colon x \mapsto \tfrac{1}{\norm{x}} x$, where $\norm{\cdot}$ denotes the Euclidean norm. Equivalently, $S$ can be understood as the map that assigns to each point $x\in\R^{n+1}\setminus \{0\}$ the equivalence class of $x$ under homothety, i.e., $x\sim y$ if there exists an $r\in \RR_+$ such that $rx=y$.

\begin{definition}
    Let $(M, \eta)$ be a contact manifold with Jacobi bracket $\{\cdot, \cdot\}$.
    A map $\hat F\colon M \to \Sp^n$ is called \emph{dissipative} at a regular value $\ray{\lambda}$ if there exists a map $F=(f_\alpha)_\alpha\colon M\to \R^{n+1}$ such that $\hat F=S\circ F$ and the functions $f_\alpha$ are in involution, namely, $\{f_\alpha, f_\beta\}=0$ for every $\alpha, \beta \in \{0, \ldots, n\}$.
\end{definition}

\begin{theorem}\label{thm:characterizations_coisotropic}
    Let $(M, \eta)$ be a contact manifold, let $\lambda \in \R^{n+1}\setminus \{0\}$. Let $F=(f_0, \dotsc, f_n)\colon M \to \R^{n+1}$ be a map such that $\operatorname{rank} \T F\geq n$ at $M_{\ray{\lambda}}=F^{-1}(\ray{\lambda})$. Then, the following statements are equivalent:
    \begin{enumerate}
        \item \label{item:coisotropic} $M_{\ray{\lambda}}$ is coisotropic.
        \item \label{item:brackets} $f_\alpha\jacBr{f_\beta,f_\gamma}_{\eta} + f_\gamma\jacBr{f_\alpha,f_\beta}_{\eta} + f_\beta\jacBr{f_\gamma,f_\alpha}_{\eta} = 0$ on $M_{\ray{\lambda}}$, for each $\alpha, \beta, \gamma \in \{0, \dotsc, n\}$.
        \item \label{item:involution} The functions $f_{\alpha}$ are in involution on $M_{\ray{\lambda}}$ namely, $\restr{\jacBr{f_\alpha,f_\beta}_{\eta}}{M_{\ray{\lambda}}} = 0$ for all $\alpha, \beta \in \{0, \dotsc, n\}$.
        \item \label{item:sphere} $\hat F=S\circ F$ is dissipative at $\ray{\lambda}$.
    \end{enumerate}
\end{theorem}

\begin{proof}
  At least one function $f_\alpha$ is non-zero in a neighbourhood of any point of $M_{\ray{\lambda}}$. Thus, without loss of generality, we can assume that $f_0$ does not vanish and define $g_i = f_i/f_0$ for $i \in \set{1,\ldots,n}$. We notice that working locally is enough and we can extend our result using a partition of unity.

  Let us start by showing that \ref{item:coisotropic} and \ref{item:brackets} are equivalent.
    Since $M_{\ray{\lambda}}$ is a level set of the $g_i$, one can apply Proposition~\ref{prop:coisotropic}, obtaining
    \begin{equation}
        0 = \JacLambda(\dd g_i, \dd g_j) =  \JacLambda \left(\dd \left(\frac{f_i}{f_0}\right), \dd \left(\frac{f_j}{f_0}\right)\right)\, ,
    \end{equation}
    on $M_{\ray{\lambda}}$.
    Hence, 
    \begin{equation}
    \begin{aligned}
        0 & = f_0^3\JacLambda \left(\dd \left(\frac{f_i}{f_0}\right), \dd \left(\frac{f_j}{f_0}\right) \right)
        = f_0\JacLambda(\dd f_i, \dd f_j) - f_i \JacLambda(\dd f_0, \dd f_j) -  f_j \JacLambda(\dd f_i, \dd f_0) 
        \\ &
        = f_0 \jacBr{f_i, f_j} -f_i  \jacBr{f_0, f_j} - f_j \jacBr{f_i, f_0}
        \, ,
    \end{aligned}
    \end{equation}
    on $M_{\ray{\lambda}}$.
    Since the choice of $f_0$ is arbitrary, we obtain that $M_{\ray{\lambda}}$ is coisotropic if and only if
    \begin{equation*}
        f_\alpha\jacBr{f_\beta,f_\gamma} + f_\gamma\jacBr{f_\alpha,f_\beta} + f_\beta\jacBr{f_\gamma,f_\alpha} = 0
    \end{equation*}
    holds on $M_{\ray{\lambda}}$.

    The next step is to prove that \ref{item:brackets} implies \ref{item:involution}. 
    By means of the symplectization, we can show that the Hamiltonian vector fields $X_{f_\alpha}$ are tangent to $M_{\ray{\lambda}}$ (see Subsection~\ref{subsec:proof}). In particular, $M_{\ray{\lambda}}$ is a level set of the functions $g_i$ and thus $X_{f_0}(g_i)=0$. Since $X_{f_0}$ is the Reeb vector field of $\eta_0 = -\eta/f_0$, we have that
    \begin{equation*}
        \jacBr{g_i, g_j}_{\eta_0} = \JacLambda_0(\dd g_i, \dd  g_j) - g_i X_{f_0}(g_j) + g_j X_{f_0}(g_i) = \JacLambda_0(\dd g_i, \dd  g_j)\, ,
    \end{equation*}
    where $(\JacLambda_0, -X_h)$ is the Jacobi structure defined by $\eta_0$. By Proposition~\ref{prop:coisotropic}, the right-hand side vanishes on $M_{\ray{\lambda}}=F^{-1}(\ray{\lambda})$.
    This in turn implies that
    \begin{equation*}
        \jacBr{f_i, f_j}_{\eta} = f_0 \jacBr{f_i/f_0, f_j/f_0}_{\eta_0} =  f_0\jacBr{g_i, g_j}_{\eta_0}
    \end{equation*}
    vanishes on $M_{\ray{\lambda}}$.
    In addition, 
    \begin{equation*}
        \jacBr{f_0, f_j}_{\eta} = f_0 \jacBr{1, f_j/f_0}_{\eta_0} =  f_0\jacBr{1, g_j}_{\eta_0} = - f_0 X_{f_0} (g_j) = 0\, ,
    \end{equation*}
    which shows that
      \begin{equation*}
        \restr{\jacBr{f_\alpha, f_\beta}_{\eta}}{M_{\ray{\lambda}}} = 0
    \end{equation*}
    for all $\alpha, \beta \in \{0, \ldots, n\}$. Obviously, \ref{item:involution} implies \ref{item:brackets}.


    Last of all, we prove the equivalence between \ref{item:involution} and \ref{item:sphere}. Let $\bar{\eta} = a \eta$. Without loss of generality, we can assume that $a > 0$ (changing the sign of the contact form would only change the sign of the brackets, but the functions would still be in involution).

    One has that $\hat{F} = S \circ  \restr{F'}{M\setminus (F')^{-1}(0)}$, where $F' = (f_0', \ldots, f_n')\colon M \to \RR^{n+1}$, if and only if $\ray{F(x)} = \ray{F'(x)}$ for all $x \in M$. Equivalently, $f_\alpha = a f'_\alpha$ for some positive function $a\in \Cinfty(M)$. Thus, $\hat{F}$ is dissipative at $\ray{\lambda}$ if and only if
    \begin{equation}
        \jacBr{f'_\alpha, f'_\beta}_{\eta} = \jacBr{a f_\alpha, a f_\beta}_{\eta} = a \jacBr{f_\alpha, f_\beta}_{\bar{\eta}}\, ,
    \end{equation}
    on $M_{\ray{\lambda}}$.
    Hence, $\hat{F}$ is dissipative if and only if there exist some contact form conformally equivalent to $\eta$ that makes the functions $f_\alpha$ be in involution. 
\end{proof}

\begin{remark}\label{rem:characterization_integrable_systems}
Let $(M, \eta)$ be a contact manifold.
Using Theorem~\ref{thm:characterizations_coisotropic}, we can give the following characterizations of a completely integrable contact system.
    \begin{enumerate}
        \item A triple $(M,\eta, F)$, where $F=(f_\alpha): M \to \RR^{n+1}$ such that $\jacBr{f_\alpha,f_\beta} = 0$ and $\T F$ has rank at least $n$ (i.e.~a completely integrable system according to Definition~\ref{def:integrable_systems}).

        \item A triple $(M,\eta, F)$, where $F=(f_\alpha): M \to \RR^{n+1}$ such that $f_\alpha\jacBr{f_\beta,f_\gamma} + f_\gamma\jacBr{f_\alpha,f_\beta} + f_\beta\jacBr{f_\gamma,f_\alpha} = 0$ at least $n$.

        \item A triple $(M, \eta, \hat{F})$, where $\hat{F}: M \to \Sp^n$ such that $\hat{F}$ is a submersion and has coisotropic fibers.
    \end{enumerate}

    It is also worth mentioning that in a previous paper~\cite{deLeon2023}, motivated by the study of complete solutions of the Hamilton--Jacobi problem, we had defined an integrable system as a contact Hamiltonian system $(M,\eta, h)$ with a foliation consisting of $(n + 1)$-dimensional coisotropic leaves invariant under the flow of the Hamiltonian vector field $X_h$. In future works, we would like to explore the relation between action-angle coordinates and complete solutions to the Hamilton--Jacobi problem.
\end{remark}

\section{Other definitions of contact integrable systems}\label{sec:other_definitions}

Several definitions of contact integrable systems can be found in the literature. Nevertheless, to the best of our knowledge, there has been no previous investigation into the concept of integrability for contact Hamiltonian systems without imposing restrictive assumptions on the dynamics. As it is well-known, in symplectic and Poisson manifolds, a collection of functions $f_1, \ldots, f_k$ are in involution if and only if their Hamiltonian vector fields $X_{f_i}$ are tangent to the level set of those functions. However, this is not the case in contact manifolds unless one assumes that every function is preserved by the flow of the Reeb vector field. This leads to dynamics that cannot describe dissipative phenomena. Instead of considering level sets of the functions in involution, we consider preimages of rays $M_{\ray{\lambda}}$, so that the contact Hamiltonian vector fields $X_{f_\alpha}$ are tangent to $M_{\ray{\lambda}}$ without needing to assume that $\Reeb (f_\alpha)=0$. Furthermore, in previous works concerning integrability of contact systems the submanifolds to which $X_{f_\alpha}$ are tangent are assumed to be compact. In several examples of contact Hamiltonian systems, e.g.~the damped harmonic oscillator, this is not the case.

\begin{itemize}
    \item Boyer~\cite{Boyer2011} introduced a concept of completely integrable system for the so-called good Hamiltonians, that is, the Hamiltonian function is preserved along the flow of the Reeb vector field. More specifically, the author says that a contact Hamiltonian system $(M, \eta, h)$ is completely integrable if there exists $n+1$ independent functions $h, f_1, \ldots, f_n$ in involution such that $X_h(h)=0$ and $X_h(f_i)=0,\, i=1, \ldots, n$. This definition implies that $\Reeb(h)=0$. This condition is overly restrictive from a dynamical systems perspective because it excludes systems with energy dissipation.
    
    \item Jovanovi\'{c}~\cite{Jovanovic2012} (see also~\cite{Jovanovic2015}) considered noncommutative integrability for the flows of contact Hamiltonian vector fields. However, the functions considered are assumed to be invariant under the Reeb flow. {In a recent preprint \cite{Jovanovic2025}, B.~Jovanovi\'{c} studies the non-commutative integrability of contact systems on a contact manifold $(M, C)$ using the Jacobi structure on the space of sections of a contact line bundle $L$. In this new work, the author no longer assumes the contact Hamiltonian to be Reeb-invariant.}
    
    \item Miranda~\cite{Miranda2014} (see also~\cite{MirandaGalceran2005}) obtained action-angle coordinates for the dynamics generated by the Reeb vector field on a contact manifold. It is worth mentioning that the Hamiltonian vector field of a nonvanishing function $f$ with respect to $\eta$ is the Reeb vector field of a conformal contact form $\bar \eta = -\eta/f$. Nevertheless, the author assumes that the Reeb vector field is an infinitesimal generator of a $\Sp^1$-action, which restricts the dynamics one can consider to periodic orbits.
    
    \item Khesin and Tabachnikov~\cite{Khesin2010} call a foliation co-Legendrian when it is transverse to $\cdist$ and $T\mathcal{F}\cap \cdist$ is integrable. Additionally, they define an integrable system as a particular case of a co-Legendrian foliation with some extra regularity conditions.
    In a previous paper, we showed that if $N$ is a $(n+1)$-dimensional co-Legendrian submanifold, then it is also a coisotropic submanifold (see Proposition~5.9 from \cite{deLeon2023}).
    


    \item Banyaga and Molino~\cite{Banyaga1993} studied completely integrable contact forms of toric type. More specifically, they considered contact forms on $M$ for which the space of first integrals of the Reeb field determines a singular fibration $\pi\colon M \to W$ defined locally by the action of a torus of dimension $n+1$ of contact transformations, where $W$ is then the space of orbits of this action. A classification of compact connected toric contact manifolds was done by Lerman~\cite{Lerman2003}.
    
    \item Geiges, Hedicke and Sa\u{g}lam~\cite{Geiges2023} recently introduced a notion of Bott integrability for Reeb flows on contact 3-manifolds. They call a Reeb flow Bott integrable if there exists a Morse--Bott function $f\colon M \to \R$ which is invariant under the Reeb flow. They proof that a closed, oriented 3-manifold admits a Bott-integrable Reeb flow if and only if it is a graph manifold.
\end{itemize}


\section{Example}\label{sec:example}


Consider the contact manifold $(M, \eta)$, with $M=\RR^3\setminus\{0\}$ and $\eta= \dd z - p \dd q$ in canonical coordinates $(q, p, z)$. The functions 
\begin{equation}
    h = p\, , \quad f = z
\end{equation}
are in involution. Indeed, their Hamiltonian vector fields are given by
\begin{equation}
    X_h = \frac{\partial}{\partial q}\, , \quad  
    X_f = -p \frac{\partial}{\partial p} - z  \frac{\partial}{\partial z}\, .
\end{equation}
Thus, 
\begin{equation}
    \{h, f\} = X_h(f) + f \Reeb(h) = 0\, .
\end{equation}
Additionally, $h$ and $f$ are functionally independent, that is, $\dd h \wedge \dd f  = \dd p \wedge \dd z \neq 0$. Hence, $(M, \eta, F)$ is a completely integrable contact system, with $F = (h, f)$. Since $F\colon (q, p, z) \mapsto (p, z)$ is the canonical projection, $F\colon \RR^3 \to \RR^2$ is a trivial bundle and the hypotheses of Theorem~\ref{thm:main_theorem} are satisfied. 

Consider the symplectization $\Sigma \colon M^\Sigma = M \times \RR_+\to M$ with $\Sigma$ the canonical projection. Let $(q, p, z, r)$ denote the bundle coordinates of $M^\Sigma$. In these coordinates, the conformal factor reads $\sigma =r$. Therefore, $\theta = r \dd z- rp \dd q$ is the symplectic potential on $M^\Sigma$, and the symplectizations of $h$ and $f$ are $h^\Sigma = -rp$ and $f^\Sigma = -rz$. Their Hamiltonian vector fields are
\begin{equation}
    X_{h^\Sigma} = \frac{\partial}{\partial q}\, , \quad  
    X_{f^\Sigma} = -p \frac{\partial}{\partial p} - z  \frac{\partial}{\partial z} + r \frac{\partial}{\partial r}\, .
\end{equation}
Consider a section $\chi\colon \RR^2 \to M^\Sigma$ of $F^\Sigma=(h^\Sigma, f^\Sigma)$ such that $\chi^\ast \theta = 0$. Suppose that $\chi$ reads
\begin{equation}
    \chi(\lambda_1, \lambda_2) = \big( \chi_q(\lambda_1, \lambda_2),  \chi_p(\lambda_1, \lambda_2),  \chi_z(\lambda_1, \lambda_2),  \chi_r(\lambda_1, \lambda_2)\big)\, .
\end{equation}
The condition of $\chi$ being a section yields
\begin{equation}
   \lambda_1 \chi_z = \lambda_2 \chi_p\, , \quad \chi_r \chi_p = \lambda_1\, ,
\end{equation}
Moreover, the condition $\chi^\ast \theta = 0$ implies that
\begin{equation}
    \dd \xi_q = \frac{\lambda_2}{\lambda_1} \dd \left(\log \left( \frac{\lambda_2}{\lambda_1} \chi_p\right)\right)\, .
\end{equation}
For instance, one can choose $\chi_r = \lambda_2/\lambda_1$ so that
\begin{equation}
    \chi(\lambda_1, \lambda_2) = \left(0, \frac{\lambda_1}{\lambda_2}, 1, \lambda_2\right)\, ,
\end{equation}
in the points where $\lambda_2\neq 0$. The Lie group action $\Phi\colon \RR^2 \times M^\Sigma \to M^\Sigma$ defined by the flows of $X_{h^\Sigma}$ and $X_{f^\Sigma}$ is given by
\begin{equation}
    \Phi(t, s; q, p, z, r) = \left(q+t, p e^{-s}, z e^{-s}, r e^{s}\right)\, ,
\end{equation}
whose isotropy subgroup is the trivial one. The angle coordinates $(y^0_\Sigma, y^1_\Sigma)$ of a point $x\in M^\Sigma$ are determined by
\begin{equation}\label{eq:flow_example}
    \Phi\left(y^0_\Sigma, y^1_\Sigma, \chi(F(x))\right) = x\, .
\end{equation}
If the canonical coordinates of $x$ are $(q, p, z, r)$, then $\chi\circ F(x) = (0, p/z, 1, rz)$ and equation~\eqref{eq:flow_example} reads
\begin{equation}
    \left(y^0_\Sigma, \frac{p}{z} e^{-y^1_\Sigma}, e^{-y^1_\Sigma}, rz e^{y^1_\Sigma}\right) = (q, p, z, r)\, .
\end{equation}
As a consequence,
\begin{equation}
    y^0_\Sigma = q\, , \quad y^1_\Sigma = - \log z\, .
\end{equation}
Since the isotropy subgroup is trivial, the action coordinates coincide with the functions in involution, namely, 
\begin{equation}
    A_0^\Sigma = h^\Sigma = -rp\, , \quad A_1^\Sigma = f^\Sigma = -rz\, .
\end{equation}
Projecting to $M$ we obtain the functions
\begin{equation}
    y^0=q\, , \quad y^1 = - \log z\, , \quad A_0 = h = p\, , \quad A_1 = f = z\, .
\end{equation}
Recall that $\chi$ is defined for points where $\lambda_2 \neq 0$, i.e., the projection by $F^\Sigma$ of points with $z\neq 0$. The action coordinate is 
\begin{equation}
    \tilde{A} = - \frac{A_0}{A_1} = -\frac{p}{z}
\end{equation}
In the coordinates $(y^0, y^1, \tilde{A})$ the Hamiltonian vector fields read
\begin{equation}
    X_h = \frac{\partial}{\partial y^0}\, , \quad X_f = \frac{\partial}{\partial y^1}\, ,
\end{equation}
and there is a conformal contact form given by
\begin{equation}
    \tilde{\eta} = - \frac{1}{A_1} \eta = \dd y^1 - \tilde{A}  \dd y^0\, .
\end{equation}
In these coordinates, contact Hamilton equations are written as
\begin{equation}
    \dot y^0 = 1\, , \quad \dot y^1 = 0\, , \quad \dot{\tilde{A}} = 0\, .
\end{equation}
Integrating them yields the curves
\begin{equation}
    c(t) = \left(y^0(t), y^1(t), \tilde{A}(t) \right) = \left(y^0(0)+t, y^1(0), \tilde{A}(0) \right)\, .
\end{equation}
{ In the coordinates $(y^0, y^1, \tilde{A})$, the Reeb vector field $\Reeb$ of $\eta$ reads
$$\Reeb = - e^{y^1} \left(\frac{\partial}{\partial y^1}+\tilde{A}\frac{\partial}{\partial \tilde{A}}\right)\, ,$$
while the Reeb vector field of $\tilde{\eta}$ is 
$$\tilde{\Reeb} = \frac{\partial}{\partial y^1}\, .$$
}
Similarly, 
\begin{equation}
    \chi(\lambda_1, \lambda_2) = \left( \frac{\lambda_2}{\lambda_1} , 1, \frac{\lambda_2}{\lambda_1}, \lambda_1 \right)
\end{equation}
is a section of $F^\Sigma$ in the points where $\lambda_1\neq 0$. Performing analogous computations as above one obtains the action-angle coordinates
\begin{equation}
    \hat{y}^0 = q - \frac{z}{p}\, , \quad \hat{y}^1 = - \log p\, , \quad \hat{A} = - \frac{z}{p}\, ,
\end{equation}
such that
\begin{equation}
    X_h = \frac{\partial}{\partial \hat{y}^0}\, , \quad X_f = \frac{\partial}{\partial \hat{y}^1}\, , \quad \hat{\eta} = - \frac{1}{p} \eta = \dd \hat{y}^0 - \hat{A} \dd \hat{y}^1\, . 
\end{equation}
{ In the coordinates $(\hat y^0, \hat y^1, \hat{A})$, the Reeb vector field $\Reeb$ of $\eta$ reads
$$\Reeb = - e^{\hat y^1} \left(\frac{\partial}{\partial \hat y^0}+\frac{\partial}{\partial \hat{A}}\right)\, .$$
}

\section*{Declaration of interest}
The authors have no competing interests to declare.



\section*{Data Availability}
Data sharing is not applicable to this article as no new data were created or analysed in this study.

\let\emph\oldemph
\printbibliography

@book{Abraham1988,
  title = {Manifolds, {{Tensor Analysis}}, and {{Applications}}},
  author = {Abraham, Ralph and Marsden, Jerrold E. and Ratiu, Tudor},
  editorb = {Marsden, J. E. and Sirovich, L. and John, F.},
  editorbtype = {redactor},
  date = {1988},
  series = {Applied {{Mathematical Sciences}}},
  volume = {75},
  publisher = {{Springer New York}},
  location = {{New York, NY}},
  doi = {10.1007/978-1-4612-1029-0},
  url = {http://link.springer.com/10.1007/978-1-4612-1029-0},
  isbn = {978-1-4612-6990-8 978-1-4612-1029-0},
  langid = {english}
}

@book{Abraham2008,
  title = {Foundations of {{Mechanics}}},
  author = {Abraham, R. and Marsden, J.E.},
  date = {2008},
  series = {{{AMS Chelsea}} Publishing},
  publisher = {{AMS Chelsea Pub./American Mathematical Society}},
  url = {https://books.google.es/books?id=4Y-ownk6ilsC},
  isbn = {978-0-8218-4438-0},
  lccn = {2008005206}
}

@online{A.E2023,
  title = {Particular Integrals and Particular Integrability for (Co)Symplectic and (Co)Contact {{Hamiltonian}} Systems},
  author = {Azuaje, R. and Escobar-Ruiz, A. M.},
  date = {2023-09-29},
  eprint = {2309.17356},
  eprinttype = {arxiv},
  eprintclass = {math-ph},
  pubstate = {preprint}
}

@book{Arnold1978,
  title = {Mathematical {{Methods}} of {{Classical Mechanics}}},
  author = {Arnold, V. I.},
  date = {1978},
  series = {Graduate {{Texts}} in {{Mathematics}}},
  publisher = {{Springer-Verlag}},
  location = {{New York}},
  doi = {10.1007/978-1-4757-1693-1},
  url = {https://www.springer.com/gp/book/9781475716931},
  isbn = {978-1-4757-1693-1},
  langid = {english}
}

@book{Audin2004,
  title = {Torus {{Actions}} on {{Symplectic Manifolds}}},
  author = {Audin, Michèle},
  date = {2004},
  publisher = {{Birkhäuser Basel}},
  location = {{Basel}},
  doi = {10.1007/978-3-0348-7960-6},
  url = {http://link.springer.com/10.1007/978-3-0348-7960-6},
  isbn = {978-3-0348-9637-5 978-3-0348-7960-6},
  langid = {english}
}

@incollection{Banyaga1993,
  title = {Géométrie Des Formes de Contact Complètement Intégrables de Type Toriques},
  booktitle = {Séminaire {{Gaston Darboux}} de {{Géométrie}} et {{Topologie Différentielle}}, 1991–1992 ({{Montpellier}})},
  author = {Banyaga, A. and Molino, P.},
  date = {1993},
  pages = {1--25},
  publisher = {{Univ. Montpellier II, Montpellier}},
  url = {https://mathscinet.ams.org/mathscinet-getitem?mr=1223155},
  mrnumber = {1223155}
}

@incollection{Banyaga1999,
  title = {The Geometry Surrounding the {{Arnold-Liouville}} Theorem},
  booktitle = {Advances in Geometry},
  author = {Banyaga, Augustin},
  date = {1999},
  series = {Progr. {{Math}}.},
  volume = {172},
  pages = {53--69},
  publisher = {{Birkhäuser Boston, Boston, MA}},
  url = {https://mathscinet.ams.org/mathscinet-getitem?mr=1667675},
  mrnumber = {1667675}
}

@book{Bolsinov2004,
  title = {Integrable {{Hamiltonian}} Systems: Geometry, Topology, Classification},
  shorttitle = {Integrable {{Hamiltonian}} Systems},
  author = {Bolsinov, A. V. and Fomenko, A. T.},
  date = {2004},
  publisher = {{Chapman \& Hall/CRC}},
  location = {{Boca Raton, Fla}},
  isbn = {978-0-415-29805-6},
  langid = {english},
  pagetotal = {730}
}

@article{Boyer2011,
  title = {Completely {{Integrable Contact Hamiltonian Systems}} and {{Toric Contact Structures}} on {{S2xS3}}},
  author = {Boyer, Charles P.},
  date = {2011-06-15},
  journaltitle = {SIGMA Symmetry Integrability Geom. Methods Appl.},
  issn = {18150659},
  doi = {10.3842/SIGMA.2011.058},
  url = {http://www.emis.de/journals/SIGMA/2011/058/}
}

@article{Bravetti2017a,
  title = {Contact {{Hamiltonian Dynamics}}: {{The Concept}} and {{Its Use}}},
  shorttitle = {Contact {{Hamiltonian Dynamics}}},
  author = {Bravetti, Alessandro},
  date = {2017-10},
  journaltitle = {Entropy},
  volume = {19},
  number = {10},
  pages = {535},
  publisher = {{Multidisciplinary Digital Publishing Institute}},
  issn = {1099-4300},
  doi = {10.3390/e19100535},
  url = {https://www.mdpi.com/1099-4300/19/10/535},
  issue = {10},
  langid = {english}
}

@article{Bravetti2019,
  title = {Contact Geometry and Thermodynamics},
  author = {Bravetti, Alessandro},
  date = {2019-02},
  journaltitle = {Int. J. Geom. Methods Mod. Phys.},
  volume = {16},
  pages = {1940003},
  issn = {0219-8878, 1793-6977},
  doi = {10.1142/S0219887819400036},
  url = {https://www.worldscientific.com/doi/abs/10.1142/S0219887819400036},
  issue = {supp01},
  langid = {english}
}

@article{Ciaglia2018,
  title = {Contact Manifolds and Dissipation, Classical and Quantum},
  author = {Ciaglia, F. M. and Cruz, H. and Marmo, G.},
  date = {2018-11-01},
  journaltitle = {Ann. Physics},
  volume = {398},
  pages = {159--179},
  issn = {0003-4916},
  doi = {10.1016/j.aop.2018.09.012},
  url = {https://www.sciencedirect.com/science/article/pii/S0003491618302574},
  langid = {english}
}

@article{Dazord1991,
  title = {Structure Locale Des Variétés de {{Jacobi}}},
  author = {Dazord, Pierre and Lichnerowicz, André and Marle, Charles-Michel},
  date = {1991},
  journaltitle = {J. Math. Pures Appl. (9)},
  volume = {70},
  number = {1},
  pages = {101--152},
  issn = {0021-7824},
  mrnumber = {1091922}
}

@article{deLeon2019,
  title = {A Review on Contact {{Hamiltonian}} and {{Lagrangian}} Systems},
  author = {family=León, given=Manuel, prefix=de, useprefix=true and Lainz, Manuel},
  date = {2019},
  journaltitle = {Revista de la Real Academia de Ciencias Canaria},
  volume = {XXXI},
  eprint = {2011.05579},
  eprinttype = {arxiv},
  pages = {1--46},
  url = {http://arxiv.org/abs/2011.05579}
}

@article{deLeon2019a,
  title = {Contact {{Hamiltonian}} Systems},
  author = {family=León, given=Manuel, prefix=de, useprefix=true and Lainz Valcázar, Manuel},
  date = {2019-10},
  journaltitle = {J. Math. Phys.},
  volume = {60},
  number = {10},
  pages = {102902},
  publisher = {{American Institute of Physics}},
  issn = {0022-2488},
  doi = {10.1063/1.5096475},
  url = {https://aip.scitation.org/doi/10.1063/1.5096475}
}

@article{deLeon2023,
  title = {Hamilton–{{Jacobi}} Theory and Integrability for Autonomous and Non-Autonomous Contact Systems},
  author = {family=León, given=Manuel, prefix=de, useprefix=true and Lainz, Manuel and López-Gordón, Asier and Rivas, Xavier},
  date = {2023-05-01},
  journaltitle = {J. Geom. Phys.},
  volume = {187},
  pages = {104787},
  issn = {0393-0440},
  doi = {10.1016/j.geomphys.2023.104787},
  url = {https://www.sciencedirect.com/science/article/pii/S0393044023000396},
  langid = {english}
}

@article{Fiorani2003,
  title = {The {{Liouville}}–{{Arnold}}–{{Nekhoroshev}} Theorem for Non-Compact Invariant Manifolds},
  author = {Fiorani, Emanuele and Giachetta, Giovanni and Sardanashvily, Gennadi},
  date = {2003-02},
  journaltitle = {J. Phys. A: Math. Gen.},
  volume = {36},
  number = {7},
  pages = {L101},
  issn = {0305-4470},
  doi = {10.1088/0305-4470/36/7/102},
  url = {https://dx.doi.org/10.1088/0305-4470/36/7/102},
  langid = {english}
}

@article{Fiorani2003a,
  title = {An Extension of the {{Liouville-Arnold}} Theorem for the Non-Compact Case},
  author = {Fiorani, E. and Giachetta, G. and Sardanashvily, G.},
  date = {2003},
  journaltitle = {Nuovo Cimento Soc. Ital. Fis. B},
  volume = {118},
  number = {3},
  pages = {307--317},
  issn = {1594-9982},
  url = {https://mathscinet.ams.org/mathscinet-getitem?mr=2162859},
  mrnumber = {2162859}
}

@book{Geiges2008,
  title = {An {{Introduction}} to {{Contact Topology}}},
  author = {Geiges, Hansjörg},
  date = {2008},
  series = {Cambridge {{Studies}} in {{Advanced Mathematics}}},
  publisher = {{Cambridge University Press}},
  location = {{Cambridge}},
  doi = {10.1017/CBO9780511611438},
  url = {https://www.cambridge.org/core/books/an-introduction-to-contact-topology/F851B2A2E7E78C6B9967A18A6641B40C},
  isbn = {978-0-521-86585-2}
}

@article{Geiges2023,
  title = {Bott-Integrable {{Reeb}} Flows on 3-Manifolds},
  author = {Geiges, Hansjörg and Hedicke, Jakob and Sağlam, Murat},
  date = {2024},
  journaltitle = {J. London Math. Soc.},
  volume = {109},
  number = {1},
  pages = {e12859},
  issn = {1469-7750},
  doi = {10.1112/jlms.12859},
  url = {https://onlinelibrary.wiley.com/doi/abs/10.1112/jlms.12859},
  langid = {english}
}

@book{Godbillon1969,
  title = {Géométrie Différentielle et Mécanique Analytique},
  author = {Godbillon, C.},
  date = {1969},
  series = {Collection {{Méthodes}}},
  publisher = {{Hermann}},
  location = {{Paris}},
  url = {https://books.google.es/books?id=0VrvAAAAMAAJ},
  lccn = {lc76420748}
}

@book{Goldstein1980,
  title = {Classical Mechanics},
  author = {Goldstein, Herbert},
  date = {1980},
  series = {Addison-{{Wesley Series}} in {{Physics}}},
  edition = {Second edition},
  publisher = {{Addison-Wesley Publishing Co., Reading, Mass.}},
  url = {https://mathscinet.ams.org/mathscinet-getitem?mr=575343},
  isbn = {978-0-201-02918-5},
  mrnumber = {575343},
  pagetotal = {xiv+672}
}

@article{Grabowska2022,
  title = {A Geometric Approach to Contact {{Hamiltonians}} and Contact {{Hamilton}}–{{Jacobi}} Theory},
  author = {Grabowska, Katarzyna and Grabowski, Janusz},
  date = {2022-11},
  journaltitle = {J. Phys. A: Math. Theor.},
  volume = {55},
  number = {43},
  pages = {435204},
  publisher = {{IOP Publishing}},
  issn = {1751-8121},
  doi = {10.1088/1751-8121/ac9adb},
  url = {https://dx.doi.org/10.1088/1751-8121/ac9adb},
  langid = {english}
}

@article{Grabowska2023a,
  title = {Reductions: Precontact versus Presymplectic},
  shorttitle = {Reductions},
  author = {Grabowska, Katarzyna and Grabowski, Janusz},
  date = {2023-06-02},
  journaltitle = {Annali di Matematica},
  issn = {1618-1891},
  doi = {10.1007/s10231-023-01341-y},
  url = {https://doi.org/10.1007/s10231-023-01341-y},
  langid = {english}
}

@article{Grabowski2013,
  title = {Brackets},
  author = {Grabowski, Janusz},
  date = {2013-08-07},
  journaltitle = {Int. J. Geom. Methods Mod. Phys.},
  volume = {10},
  number = {8},
  publisher = {{World Scientific Publishing Company}},
  doi = {10.1142/S0219887813600013},
  url = {https://www.worldscientific.com/doi/epdf/10.1142/S0219887813600013},
  langid = {english}
}

@article{Ibanez1997,
  title = {Co-Isotropic and {{Legendre}} - {{Lagrangian}} Submanifolds and Conformal {{Jacobi}} Morphisms},
  author = {Ibáñez, Raúl and family=León, given=Manuel, prefix=de, useprefix=false and Marrero, Juan C. and family=Diego, given=David Martín, prefix=de, useprefix=false},
  date = {1997-08},
  journaltitle = {J. Phys. A: Math. Gen.},
  volume = {30},
  number = {15},
  pages = {5427},
  issn = {0305-4470},
  doi = {10.1088/0305-4470/30/15/027},
  url = {https://dx.doi.org/10.1088/0305-4470/30/15/027},
  langid = {english}
}

@article{Jovanovic2012,
  title = {Noncommutative Integrability and Action-Angle Variables in Contct Geometry},
  author = {Jovanović, Božidar},
  date = {2012-12},
  journaltitle = {J. Symplectic Geom.},
  volume = {10},
  number = {4},
  pages = {535--561},
  publisher = {{International Press of Boston}},
  issn = {1527-5256, 1540-2347},
  url = {https://projecteuclid.org/journals/journal-of-symplectic-geometry/volume-10/issue-4/Noncommutative-integrability-and-action-angle-variables-in-contct-geometry/jsg/1357153428.full}
}

@article{Jovanovic2015,
  title = {Contact Flows and Integrable Systems},
  author = {Jovanović, Božidar and Jovanović, Vladimir},
  date = {2015-01-01},
  journaltitle = {J. Geom. Phys.},
  series = {Finite Dimensional Integrable Systems: On the Crossroad of Algebra, Geometry and Physics},
  volume = {87},
  pages = {217--232},
  issn = {0393-0440},
  doi = {10.1016/j.geomphys.2014.07.030},
  url = {https://www.sciencedirect.com/science/article/pii/S0393044014001715},
  langid = {english}
}

@article{Khesin2010,
  title = {Contact Complete Integrability},
  author = {Khesin, B. and Tabachnikov, S.},
  date = {2010-10},
  journaltitle = {Regul. Chaotic Dyn.},
  volume = {15},
  number = {4-5},
  pages = {504--520},
  issn = {1560-3547, 1468-4845},
  doi = {10.1134/S1560354710040076},
  url = {http://link.springer.com/10.1134/S1560354710040076},
  langid = {english}
}

@book{K.N1996,
  title = {Foundations of Differential Geometry. {{Vol}}. {{I}}},
  author = {Kobayashi, Shoshichi and Nomizu, Katsumi},
  date = {1996},
  series = {Wiley {{Classics Library}}},
  publisher = {{John Wiley \& Sons, Inc., New York}},
  isbn = {978-0-471-15733-5},
  mrnumber = {1393940},
  pagetotal = {xii+329}
}

@thesis{Lainz2022,
  type = {phdthesis},
  title = {Contact {{Hamiltonian Systems}}},
  author = {Lainz, Manuel},
  date = {2022},
  institution = {{Universidad Autónoma de Madrid}},
  url = {http://hdl.handle.net/10486/704774},
  langid = {english}
}

@article{Laurent-Gengoux2010,
  title = {Action-Angle {{Coordinates}} for {{Integrable Systems}} on {{Poisson Manifolds}}},
  author = {Laurent-Gengoux, C. and Miranda, E. and Vanhaecke, P.},
  date = {2010-07-13},
  journaltitle = {Int. Math. Res. Not.},
  pages = {rnq130},
  issn = {1073-7928, 1687-0247},
  doi = {10.1093/imrn/rnq130},
  url = {https://academic.oup.com/imrn/article-lookup/doi/10.1093/imrn/rnq130},
  langid = {english}
}

@article{Le2018,
  title = {Deformations of Coisotropic Submanifolds in {{Jacobi}} Manifolds},
  author = {Lê, Hông Vân and Oh, Yong-Geun and Tortorella, Alfonso G. and Vitagliano, Luca},
  date = {2018},
  journaltitle = {J. Symplectic Geom.},
  volume = {16},
  number = {4},
  pages = {1051--1116},
  issn = {1527-5256},
  doi = {10.4310/JSG.2018.v16.n4.a7},
  url = {https://mathscinet.ams.org/mathscinet-getitem?mr=3917729},
  mrnumber = {3917729}
}

@article{Lerman2003,
  title = {Contact Toric Manifolds},
  author = {Lerman, Eugene},
  date = {2003},
  journaltitle = {J. Symplectic Geom.},
  volume = {1},
  number = {4},
  pages = {785--828},
  issn = {1527-5256},
  url = {https://mathscinet.ams.org/mathscinet-getitem?mr=2039164},
  mrnumber = {2039164}
}

@book{Libermann1987,
  title = {Symplectic {{Geometry}} and {{Analytical Mechanics}}},
  author = {Libermann, Paulette and Marle, Charles-Michel},
  date = {1987},
  publisher = {{Springer Netherlands}},
  location = {{Dordrecht}},
  doi = {10.1007/978-94-009-3807-6},
  url = {http://link.springer.com/10.1007/978-94-009-3807-6},
  isbn = {978-90-277-2439-7 978-94-009-3807-6},
  langid = {english}
}

@article{Libermann1991,
  title = {Legendre Foliations on Contact Manifolds},
  author = {Libermann, Paulette},
  date = {1991},
  journaltitle = {Differential Geom. Appl.},
  volume = {1},
  number = {1},
  pages = {57--76},
  issn = {0926-2245},
  doi = {10.1016/0926-2245(91)90022-2},
  url = {https://mathscinet.ams.org/mathscinet-getitem?mr=1109814},
  mrnumber = {1109814}
}

@inproceedings{Michor1987,
  title = {Remarks on the {{Schouten-Nijenhuis}} Bracket},
  booktitle = {Rend. {{Circ}}. {{Mat}}. {{Palermo}} (2) {{Suppl}}.},
  author = {Michor, Peter W.},
  date = {1987},
  series = {Rendiconti Del {{Circolo Matematico}} Di {{Palermo}}},
  number = {16},
  pages = {207--215},
  issn = {1592-9531},
  fjournal = {Rendiconti del Circolo Matematico di Palermo. Serie II. Supplemento},
  mrclass = {58A10 (53A55)},
  mrnumber = {946726},
  mrreviewer = {Josef Janyška}
}

@article{Miranda2014,
  title = {Integrable Systems and Group Actions},
  author = {Miranda, Eva},
  date = {2014-02-01},
  journaltitle = {Open Mathematics},
  volume = {12},
  number = {2},
  pages = {240--270},
  publisher = {{De Gruyter Open Access}},
  issn = {2391-5455},
  doi = {10.2478/s11533-013-0333-6},
  url = {https://www.degruyter.com/document/doi/10.2478/s11533-013-0333-6/html},
  langid = {english}
}

@online{Miranda2022,
  title = {Action-Angle Coordinates and {{KAM}} Theory for Singular Symplectic Manifolds},
  author = {Miranda, Eva and Planas, Arnau},
  date = {2022-12-31},
  eprint = {2301.00266},
  eprinttype = {arxiv},
  eprintclass = {math},
  pubstate = {preprint}
}

@thesis{MirandaGalceran2003,
  type = {phdthesis},
  title = {On symplectic linearization of singular lagrangian foliations},
  author = {Miranda, Eva},
  date = {2003},
  institution = {{Universitat de Barcelona}},
  url = {https://documat.unirioja.es/servlet/tesis?codigo=19039},
  langid = {spanish}
}

@inproceedings{MirandaGalceran2005,
  title = {A Normal Form Theorem for Integrable Systems on Contact Manifolds},
  booktitle = {Proceedings of {{XIII Fall Workshop}} on {{Geometry}} and {{Physics}}},
  author = {Miranda, Eva},
  date = {2005},
  pages = {240--246},
  publisher = {{Real Sociedad Matemática Española}},
  location = {{Murcia, Spain}},
  url = {https://dialnet.unirioja.es/servlet/articulo?codigo=2174084},
  isbn = {978-84-933610-6-8},
  langid = {english}
}

@online{Montaldi2023,
  title = {Equilibria and Bifurcations in Contact Dynamics},
  author = {Montaldi, James},
  date = {2023-10-01},
  eprint = {2310.00764},
  eprinttype = {arxiv},
  eprintclass = {math-ph},
  pubstate = {preprint}
}

@article{Pang1990,
  title = {The Structure of {{Legendre}} Foliations},
  author = {Pang, Myung-Yull},
  date = {1990},
  journaltitle = {Trans. Amer. Math. Soc.},
  volume = {320},
  number = {2},
  pages = {417--455},
  issn = {0002-9947, 1088-6850},
  doi = {10.1090/S0002-9947-1990-1016808-6},
  url = {https://www.ams.org/tran/1990-320-02/S0002-9947-1990-1016808-6/},
  langid = {english}
}

@book{Vaisman1994,
  title = {Lectures on the {{Geometry}} of {{Poisson Manifolds}}},
  author = {Vaisman, Izu},
  date = {1994},
  publisher = {{Birkhäuser Basel}},
  location = {{Basel}},
  doi = {10.1007/978-3-0348-8495-2},
  url = {http://link.springer.com/10.1007/978-3-0348-8495-2},
  isbn = {978-3-0348-9649-8 978-3-0348-8495-2},
  langid = {english}
}

@book{Wiggins2003,
  title = {Introduction to Applied Nonlinear Dynamical Systems and Chaos},
  author = {Wiggins, Stephen},
  date = {2003},
  series = {Texts in Applied Mathematics},
  edition = {2nd ed},
  number = {2},
  publisher = {{Springer}},
  location = {{New York}},
  isbn = {978-0-387-00177-7},
  langid = {english},
  pagetotal = {843}
}

@online{Zung2012,
  title = {Action-{{Angle}} Variables on {{Dirac}} Manifolds},
  author = {Zung, Nguyen Tien},
  date = {2012-04-17},
  eprint = {1204.3865},
  eprinttype = {arxiv},
  eprintclass = {math-ph},
  pubstate = {preprint}
}

@article{G.L1984,
  title = {Géométrie Des Algèbres de {{Lie}} Locales de {{Kirillov}}},
  author = {Guedira, Fouzia and Lichnerowicz, André},
  date = {1984},
  journaltitle = {J. Math. Pures Appl. (9)},
  volume = {63},
  number = {4},
  pages = {407--484},
  issn = {0021-7824},
  mrnumber = {789560}
}

@article{Kirillov1976,
  title = {Local {{Lie Algebras}}},
  author = {Kirillov, A. A.},
  date = {1976-08-31},
  journaltitle = {Russ. Math. Surv.},
  volume = {31},
  number = {4},
  pages = {55--75},
  issn = {0036-0279, 1468-4829},
  doi = {10.1070/RM1976v031n04ABEH001556},
  url = {http://stacks.iop.org/0036-0279/31/i=4/a=R02?key=crossref.b88697572b48af4d7d2216e575131451}
}

@article{Lichnerowicz1977,
  title = {Variétés de {{Jacobi}} et Algèbres de {{Lie}} Associées},
  author = {Lichnerowicz, André},
  date = {1977},
  journaltitle = {C. R. Acad. Sci. Paris Sér. A-B},
  volume = {285},
  number = {6},
  pages = {A455--A459},
  issn = {0151-0509},
  fjournal = {Comptes Rendus Hebdomadaires des Séances de l'Académie des Sciences. Séries A et B},
  mrclass = {58F05},
  mrnumber = {455037}
}

@article{Lichnerowicz1977a,
  title = {Les Variétés de {{Poisson}} et Leurs Algèbres de {{Lie}} Associées},
  author = {Lichnerowicz, André},
  date = {1977},
  journaltitle = {J. Differential Geometry},
  volume = {12},
  number = {2},
  pages = {253--300},
  issn = {0022-040X,1945-743X},
  url = {http://projecteuclid.org/euclid.jdg/1214433987},
  fjournal = {Journal of Differential Geometry},
  mrclass = {58F05},
  mrnumber = {501133},
  mrreviewer = {A. Crumeyrolle}
}

@article{Lichnerowicz1978,
  title = {Les Variétés de {{Jacobi}} et Leurs Algèbres de {{Lie}} Associées},
  author = {Lichnerowicz, André},
  date = {1978},
  journaltitle = {J. Math. Pures Appl. (9)},
  volume = {57},
  number = {4},
  pages = {453--488},
  issn = {0021-7824},
  url = {https://mathscinet.ams.org/mathscinet-getitem?mr=524629},
  mrnumber = {524629}
}

@article{B.G.M+2024,
  title = {Kirillov Structures and Reduction of {{Hamiltonian}} Systems by Scaling and Standard Symmetries},
  author = {Bravetti, A. and Grillo, S. and Marrero, J. C. and Padrón, E.},
  date = {2024},
  journaltitle = {Studies in Applied Mathematics},
  volume = {153},
  number = {1},
  pages = {e12681},
  issn = {1467-9590},
  doi = {10.1111/sapm.12681},
  url = {https://onlinelibrary.wiley.com/doi/abs/10.1111/sapm.12681},
  langid = {english}
}

@article{B.J.S2023,
  title = {Scaling Symmetries, Contact Reduction and {{Poincaré}}’s Dream},
  author = {Bravetti, Alessandro and Jackman, Connor and Sloan, David},
  date = {2023-10},
  journaltitle = {J. Phys. A: Math. Theor.},
  volume = {56},
  number = {43},
  pages = {435203},
  publisher = {IOP Publishing},
  issn = {1751-8121},
  doi = {10.1088/1751-8121/acfddd},
  url = {https://dx.doi.org/10.1088/1751-8121/acfddd},
  langid = {english}
}

@article{G.G.R2022,
  title = {{{VB-structures}} and Generalizations},
  author = {Grabowska, Katarzyna and Grabowski, Janusz and Ravanpak, Zohreh},
  date = {2022},
  journaltitle = {Ann. Global Anal. Geom.},
  volume = {62},
  number = {1},
  pages = {235--284},
  issn = {0232-704X,1572-9060},
  doi = {10.1007/s10455-022-09847-z},
  mrnumber = {4434694}
}

@online{G.G2024,
  title = {Homogeneity Supermanifolds and Homogeneous {{Darboux}} Theorem},
  author = {Grabowska, Katarzyna and Grabowski, Janusz},
  date = {2024-11-01},
  eprint = {2411.00537},
  eprinttype = {arXiv},
  doi = {10.48550/arXiv.2411.00537},
  url = {http://arxiv.org/abs/2411.00537},
  pubstate = {prepublished}
}

@article{G.G2024a,
  title = {Contact Geometric Mechanics: The {{Tulczyjew}} Triples},
  shorttitle = {Contact Geometric Mechanics},
  author = {Grabowska, Katarzyna and Grabowski, Janusz},
  date = {2024-09-13},
  journaltitle = {Adv. Theor. Math. Phys.},
  volume = {28},
  number = {2},
  pages = {599--654},
  publisher = {International Press of Boston},
  issn = {1095-0753},
  doi = {10.4310/ATMP.240914022224},
  url = {https://link.intlpress.com/JDetail/1834660029902925826},
  langid = {english}
}

@online{Jovanovic2025,
  title = {Contact Line Bundles, Foliations, and Integrability},
  author = {Jovanovic, Bozidar},
  date = {2025-02-05},
  eprint = {2502.02935},
  eprinttype = {arXiv},
  eprintclass = {math},
  doi = {10.48550/arXiv.2502.02935},
  url = {http://arxiv.org/abs/2502.02935},
  pubstate = {prepublished},
  version = {1}
}

\end{document}